\documentclass[11pt,english]{article}
 \usepackage[english]{babel}
\usepackage{cite} 

\usepackage{graphicx,subfigure}

\usepackage[usenames, dvipsnames]{xcolor}
\usepackage[utf8x]{inputenc}
\usepackage[T1]{fontenc}
\usepackage{tikz}
\usepackage{pgfplots}

\usepackage{amsmath}
\usepackage{amsfonts}
\usepackage{amssymb}
\usepackage{amsthm}

\usepackage{esint}

\usepackage[colorlinks=true,urlcolor=blue,
citecolor=red,linkcolor=blue,linktocpage,pdfpagelabels,
bookmarksnumbered,bookmarksopen,backref=page]{hyperref}

\DeclareMathOperator{\Span}{span}
 \DeclareMathOperator{\sign}{sign}

 \newcommand{\IM}{\mathrm{Im}\,}
 \newcommand{\RE}{\mathrm{Re}\,}
 
\newtheorem{theorem}{Theorem}[section]

 \newtheorem{proposition}[theorem]{Proposition}
 \theoremstyle{definition}
 \newtheorem{definition}[theorem]{Definition}
 \theoremstyle{remark}
 \newtheorem{remark}[theorem]{Remark}
 
 \numberwithin{equation}{section}

\newcommand{\R}{\mathbb R}
\begin{author}
{Jaime Angulo Pava $^1$ and Ram\'on G. Plaza $^2$}
\end{author}

\begin{title}{Unstable kink and anti-kink profile  for the sine-{G}ordon equation on a $\mathcal{Y}$-junction graph with $\delta'$-interaction at the vertex}
\end{title}
\begin{document}
\maketitle

\centerline{$^1$ Department of Mathematics,
IME-USP}
 \centerline{Rua do Mat\~ao 1010, Cidade Universit\'aria, CEP 05508-090,
 S\~ao Paulo, SP (Brazil)}
 \centerline{\tt angulo@ime.usp.br}
 \centerline{ $^2$ Instituto de Investigaciones en Matem\'aticas Aplicadas
   y en Sistemas,}
 \centerline{Universidad Nacional Aut\'{o}noma de M\'{e}xico,  Circuito Escolar s/n,}
 \centerline{Ciudad Universitaria, C.P. 04510 Cd. de M\'{e}xico (Mexico)}
 \centerline{\tt  plaza@mym.iimas.unam.mx}

\begin{abstract}
%
The sine-Gordon equation on a metric graph with a structure represented by a $\mathcal Y$-junction, is considered. The model is endowed with boundary conditions at the graph-vertex of $\delta'$-interaction type, expressing continuity of the derivatives of the wave functions plus a Kirchhoff-type rule for the self-induced magnetic flux. It is shown that particular stationary, kink and kink/anti-kink soliton profile solutions to the model are linearly (and nonlinearly) unstable. To that end, a recently developed linear instability criterion for evolution models on metric graphs by Angulo and Cavalcante (2020), which provides the sufficient conditions on the linearized operator around the wave to have a pair of real positive/negative eigenvalues, is applied. This leads to the spectral study to the linearize operator and of its Morse index. The analysis is based on analytic perturbation theory, Sturm-Liouville oscillation results and the extension theory of symmetric operators. The  methods presented in this manuscript   have prospect for the study of the dynamic of solutions for the sine-Gordon model on metric graphs  with  finite bounds  or on metric tree graphs and/or loop graphs.
\vskip0.1in

\end{abstract}

\textbf{Mathematics  Subject  Classification (2010)}. Primary
35Q51, 35Q53, 35J61; Secondary 47E05.\\

\textbf{Key  words}.  sine-Gordon model, metric graphs, tail and bump solutions, $\delta'$-type interaction, perturbation theory, extension theory, instability.


\section{Introduction}

The one-dimensional sine-Gordon equation in laboratory coordinates,
\begin{equation}
\label{sine-G}
u_{tt} - c^2 u_{xx} + \sin u = 0,
\end{equation}
where $c > 0$ is a constant and $x \in \R$, $t > 0$, is ubiquitous in a great variety of physical and biological models. For example, it has been used to describe the magnetic flux in a long Josephson line in superconductor theory \cite{BEMS,BaPa82,SCR,Jsph65}, mechanical oscillations of a nonlinear pendulum \cite{Dra83,Knob00} and the dynamics of a crystal lattice near a dislocation \cite{FreKo}. Recently, soliton solutions to equation \eqref{sine-G} have been used as simplified models of scalar gravitational fields in general relativity theory \cite{CFMT19,FCT18} and of oscillations describing the dynamics of DNA chains \cite{DerGa11,IvIv13} in the context of the \emph{solitons in DNA hypothesis} \cite{EKHKL80}. In addition to its wide applicability, the sine-Gordon equation \eqref{sine-G} underlies many remarkable mathematical features such as a Hamiltonian structure \cite{TaFa76}, complete integrability \cite{AKNS73,AKNS74} and the existence of localized solutions (solitons) \cite{SCM,Sco03}. 

In a recent contribution \cite{AnPl-delta}, we performed the first rigorous analytical study of the stability properties of stationary soliton solutions to the sine-Gordon equation \eqref{sine-G} \emph{posed on a $\mathcal{Y}$-junction metric graph}. A metric graph is a network-shaped structure of edges which are assigned a length and connected at vertices according to boundary conditions which determine the dynamics on the network. A $\mathcal{Y}$-junction is a particular graph with three edges connected through one single vertex. There exist two main types of $\mathcal{Y}$-junctions. A $\mathcal{Y}$-junction of the first type (or type I) consists of one incoming (or parent) edge, $E_1 = (-\infty,0)$, meeting at one vertex located at the origin, $\nu = 0$, with other two outgoing (children) edges, $E_j = (0,\infty)$, $j = 2,3$. The second type (or $\mathcal{Y}$-junction of type II) is constituted by three identical edges of the form $E_j = (0,\infty)$, $1 \leqq j \leqq 3$; they are often referred to as \emph{tricrystal junctions or star graph}. See Figure \ref{figYjunction} for an illustration. Recently, junctions of type I have been used in the description of unidirectional fluid flow models (see, for example, \cite{AngCav, AngCav2, BoCa08}) or in the modeling of Josephson superconductor junctions \cite{Grunn93,Susa19}, whereas $\mathcal{Y}$-junctions of  red type II appear in the study of the nonlinear Schr\"odinger equation on graphs (see, for example,  \cite{AngGol18a,AngGol18b} and the reference cited therein)  or in the description of Josephson vortices in crystal's theory (see, \cite{SvG05,KCK00}). Recently, in \cite{Sabi18} has been studied  the stationary solutions for the sine-Gordon on star graph with a $\mathcal{Y}$ configuration and with finite bounds, namely, $E_j = (0,L_j)$, $1 \leqq j \leqq 3$, or on metric tree graphs consisting of finite bonds (see Figure \ref{figtree}).  
\begin{figure}[t]
\begin{center}
\subfigure[$\mathcal{Y} = (-\infty,0) \cup (0,\infty) \cup (0,\infty)$]{\label{figYtipoI}\includegraphics[scale=.3, clip=true]{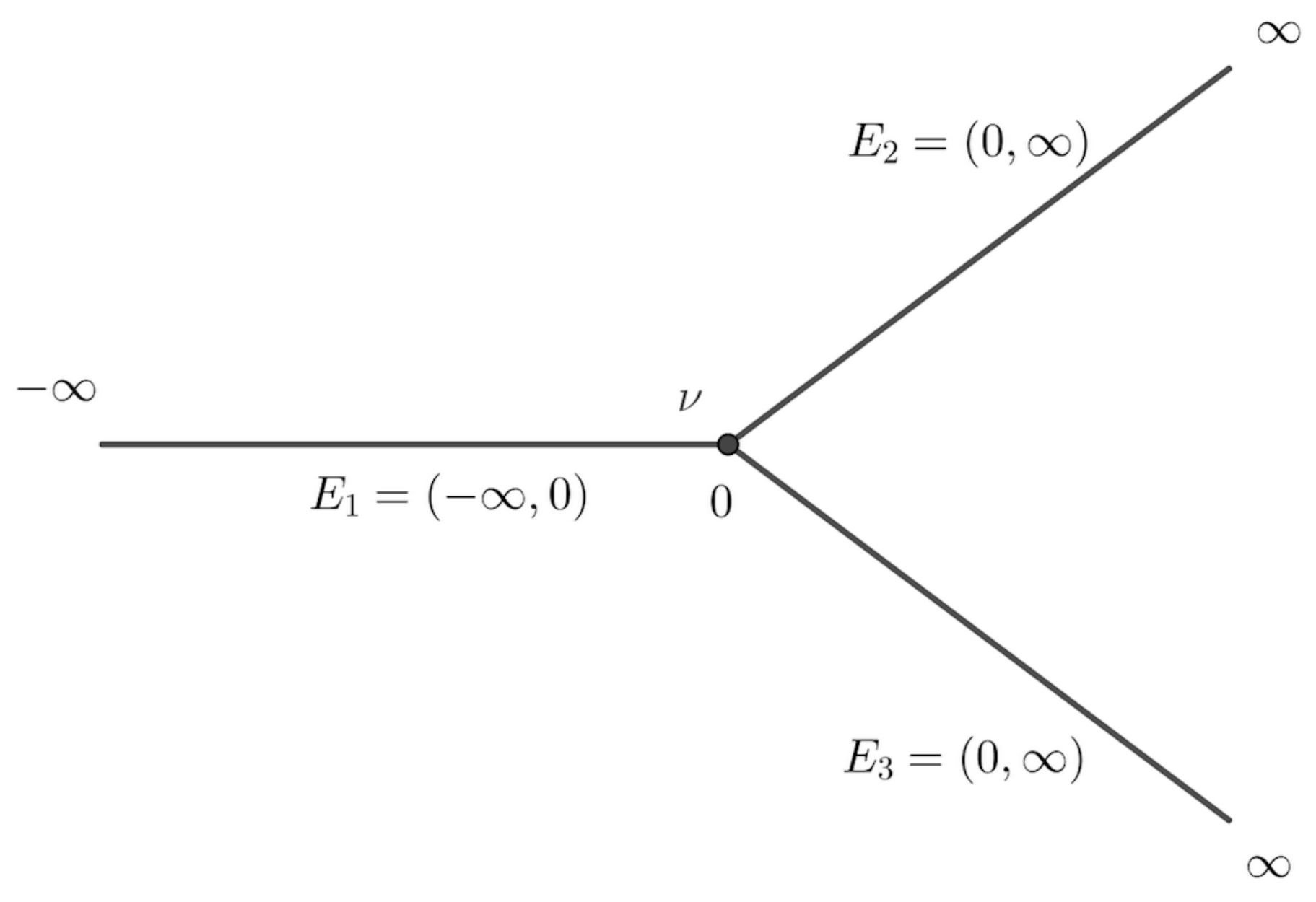}}
\subfigure[$\mathcal{Y} = (0,\infty) \cup (0,\infty) \cup (0,\infty)$]{\label{figYtipoII}\includegraphics[scale=.3, clip=true]{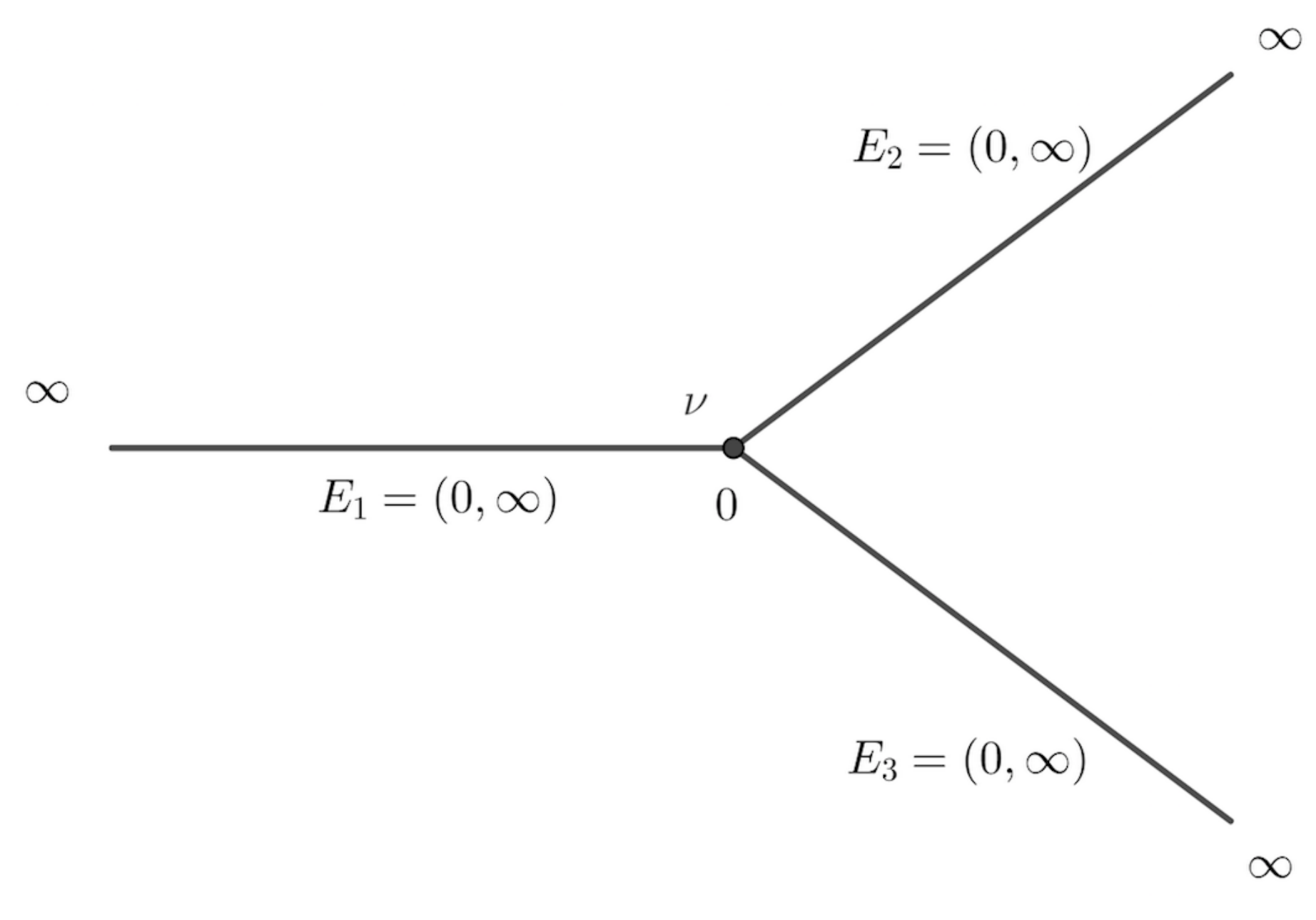}}
\end{center}
\caption{\small{Panel (a) shows a $\mathcal{Y}$-junction of the first type with $E_1 = (-\infty,0)$ and $E_j = (0,\infty)$, $j=2,3$, whereas panel (b) shows a $\mathcal{Y}$-junction of the second type (star graph or tricrystal junction) with $E_j = (0,\infty)$, $1 \leqq j \leqq 3$.}}\label{figYjunction}
\end{figure}

In the aforementioned previous paper \cite{AnPl-delta}, we studied the sine-Gordon equation \eqref{sine-G} posed on a $\mathcal{Y}$-junction of either type, and in the case where the dynamics is determined by boundary conditions at the vertex of $\delta$-type. Interactions of $\delta$-type refer to continuity of the wave functions at the vertex together with a balance flux relation for the derivatives of the wave functions (see \cite{AnPl-delta,Caput14,DuCa18,Sabi18} and the references cited therein). Motivated by physical applications, the purpose of the present paper is to study the stability of particular stationary solutions to the sine-Gordon equation posed on a $\mathcal{Y}$-graph endowed with \emph{interactions at the vertex of $\delta'$-type}. Indeed, in the context of superconductor theory, the sine-Gordon equation on a metric graph arises as a model for coupling of two or more Josephson junctions in a network. A Josephson junction is a quantum mechanical structure that is made by two superconducting electrodes separated by a barrier (the junction), thin enough to allow coupling of the wave functions of electrons for the two superconductors \cite{Jsph65}. After appropriate normalizations, it can be shown that the phase difference $u$ (also known as order parameter) of the two wave functions satisfies the sine-Gordon equation \eqref{sine-G} \cite{Jsph65,BEMS}. Coupling three junctions at one common vertex, the so called tricrystal junction, can be regarded (and fabricated) as a probe of the order parameter symmetry of high temperature superconductors (cf. \cite{Tsuei94,Tsuei00}). Physically coupling three otherwise independent long Josephson junctions, $\mathcal{Y} = \cup_{j=1}^3 E_j$, together at one common vertex, was first proposed by Nakajima \emph{et al.} \cite{NakO76,NakO78} as a prototype for logic circuits. In this framework, the sine-Gordon model in a $\mathcal{Y}$-junction is
\begin{equation}
\label{sgY}
\partial_t^2 u_j - c_j^2 \partial_x^2 u_j + \sin u_j= 0, \qquad x \in E_j, \;\,  t > 0, \;\, j=1,2,3,
\end{equation}
where $u_j$ denotes the phase difference for the magnetic flux on each edge, $E_j$. Since the surface current density should be the same in all three films at the vertex, Nakajima \emph{et al.} \cite{NakO76,NakO78} (see also \cite{Grunn93,KCK00}) impose the condition
\begin{equation}
\label{bcmagnet}
c_1 \partial_x {u_1}_{|x=0} = c_2 \partial_x {u_2}_{|x=0} = c_3 \partial_x {u_3}_{|x=0},
\end{equation}
expressing that the magnetic field, which is proportional to the derivative of phase difference, should be continuous at the intersection. Moreover, the magnetic flux computed along an infinitesimal small contour encircling the origin (vertex) must vanish, that is, the total change of the gauge invariant phase difference must be zero \cite{KCK00,Susa19}. This leads to the Kirchhoff-type of boundary condition
\begin{equation}
\label{Kirchhoffbc}
\begin{aligned}
-c_1  u_1(0-) + \sum_{j=2}^3 c_j  u_j(0+) &= 0, & & \text{($\mathcal{Y}$-junction of type I),}\\
\sum_{j=1}^3 c_j  u_j(0+) &= 0, & & \text{($\mathcal{Y}$-junction of type II).}
\end{aligned}
\end{equation}
The interaction conditions \eqref{bcmagnet} and \eqref{Kirchhoffbc} are known as boundary conditions of $\delta'$-type: they express continuity of the fluxes (derivatives) plus a Kirchhoff-type rule for the self-induced magnetic flux. 

The first study of static soliton-type (kink or anti-kink) solutions in tricrystal junctions under $\delta'$-conditions is due to Grunnet-Jepsen \emph{et al.} \cite{Grunn93}, and later pursued by other authors (for an abridged list of references, see \cite{KCK00,SvG05}). A recent work \cite{Susa19} considers solutions of \emph{breather} type as well. Up to our knowledge, however, there are no rigorous analytical studies of the stability of stationary solutions to the sine-Gordon model on a graph with boundary conditions of $\delta'$-interaction type available in the literature.  Our principal interest here will be study the stability properties for kink or kink/anti-kink solutions (see Figures  \ref{figKsolitons} and \ref{figAKsolitons}) for the sine-Gordon model on a $\mathcal{Y}$-junction of type I under $\delta'$-conditions at the vertex. Indeed, we show that they are linearly and nonlinearly unstable profiles (see subsection 1.2.3 and Theorem \ref{2main}). The stability studied of static configurations for the sine-Gordon model in the case of tricrystal junctions with infinite or finite bounds,  or on metric tree graphs (see Figure \ref{figtree}) and/or loop graphs, it will be  the focus of a future work.  We call the attention that the study of these static configurations is an important issue from both the mathematical and the physical viewpoints (see \cite{Sabi18}). 

In the stability analysis, it is customary to linearize the equation around the profile solution and to obtain useful information from the spectral properties of the linearized operator posed on an appropriate function space. Upon linearization of the sine-Gordon equation \eqref{sine-G} around a stationary soliton solution, we end up with a Schr\"odinger type operator with a bounded potential (see the form of the operator \eqref{trav23} below) that can be appropriately defined on a graph. Therefore, we adopt a quantum-graph approach \cite{AngCav2, Berko17, BlaExn08}  in order to make precise the boundary conditions that provide self-adjoint extensions on the graph of the symmetric Schr\"odinger type operator and that actually determine the physical model.

\begin{figure}[h]
\centering
\includegraphics[angle=0, scale=0.4]{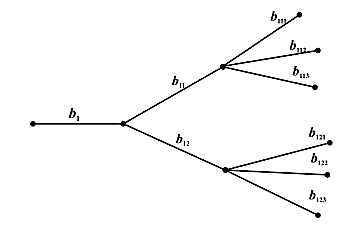}
\caption{\small{A metric tree graph with finite bounds determined by $b_1$ (father), $b_{11}, b_{12}$ (sons), $b_{111}, b_{112}, b_{113}$, $b_{121}, b_{122}, b_{123}$ (grandsons).}}
\label{figtree}
\end{figure}

\subsection{Boundary conditions of $\delta'$-interaction type}

In the case of a $\mathcal{Y}$-junction of type I (it which will be the focus of our study), the transition conditions of $\delta'$-type have the form
\begin{equation}
\label{bcI}
c_1 u'_1(0-) = c_2 u'_2(0+) = c_3 u'_3(0+),\quad
-c_1  u_1(0-) + \sum_{j=2}^3 c_j  u_j(0+) = \lambda c_1 u'_1(0-),
\end{equation}
where $\lambda \in \R$ is a parameter. The reason for considering this boundary condition is the fact that all the self-adjoint extensions of the  following symmetric operator 
\[
\mathcal{H} \mathbf{u} = \left\{ \Big( -c_j^2 \frac{d^2}{dx^2} \Big) u_j \right\}_{j=1}^3, \qquad \mathbf{u} = (u_j)_{j=1}^3,
\]
posed on a $\mathcal{Y}$-junction with the domain (see Proposition \ref{2M} in Appendix \S A)
\begin{equation*}
D(\mathcal{H})= \Big \{(u_j)_{j=1}^3\in H^2(\mathcal{Y}):
c_1u_1'(0-)=c_2u'_2(0+)=c_3u_3'(0+)=0,\;\; \sum_{j=2}^3 c_ju_j(0+)-c_1u_1(0-)=0 \Big \},
\end{equation*}
red they are defined by the boundary  conditions \eqref{bcI} which are compatible with the flux continuity condition \eqref{bcmagnet} ( for convenience of the reader, we provide a direct proof of this fact in Appendix \S \ref{secAppET}). Notice that we recover the Kirchhoff boundary condition \eqref{Kirchhoffbc} when $\lambda = 0$. These conditions depend upon the parameter $\lambda$, which ranges along the whole real line. We note that the value $\lambda \in \R$ is part of the parameters that determine the physical model (such as the speeds $c_j$, for example). Instead of adopting \emph{ad hoc} boundary conditions, we consider a parametrized family of transition rules covering a wide range of applications and which, for the particular value $\lambda = 0$, include the Kirchhoff condition \eqref{Kirchhoffbc} previously studied in the literature. Our goal is to study particular solutions to the sine-Gordon equation on the graph, subjected to boundary conditions \eqref{bcI} and motivated by the well-known \emph{kink} (or \emph{anti-kink}) solutions to the sine-Gordon equation on the real line (also referred to as topological solitons \cite{Dra83,Sco03,SCM}).

\subsection{Main results}


In this paper we consider the sine-Gordon equation \eqref{sine-G} on a metric graph with the shape of a $\mathcal{Y}$-junction with three semi-infinite edges and joined by a single vertex $\nu = 0$. In the sequel we assume that the $\mathcal{Y}$-junction is  of type I, where $E_1 = (-\infty,0)$ and $E_j = (0,\infty)$, $j = 2,3$.  The sine-Gordon model on a $\mathcal{Y}$-junction reads
\begin{equation}
\label{sg1}
\partial_t^2 u_j - c_j^2 \partial_x^2 u_j + \sin u_j= 0, \qquad x \in E_j, \;\,  t > 0, \;\, 1 \leqq j \leqq 3,
\end{equation}
where $\mathbf{u} = (u)_{j=1}^3$, $u_j = u_j(x,t)$. It is assumed that the characteristic speed on each edge $E_j$ is constant and positive, $c_j > 0$, without loss of generality. Clearly, one can recast the equations in \eqref{sg1} as a first order system that reads
\begin{equation}
\label{sg2}
\begin{cases}
\partial_t u_j = v_j\\
\partial_t v_j =c_j^2 \partial_x^2 u_j - \sin u_j,
\end{cases}
\qquad x \in E_j, \;\,  t > 0, \;\, 1 \leqq j \leqq 3.
\end{equation}
Moreover, one can always rewrite system \eqref{sg2} in the vectorial form 
  \begin{equation}\label{stat1}
  \bold w_t=JE\bold w +F(\bold w)
 \end{equation} 
  where $\bold w=(u, v)^\top$, with $u=(u_1, u_2, u_3)^\top$, $v=(v_1, v_2, v_3)^\top$, $u_j, v_j: E_j \to \mathbb R$, $1 \leqq j \leqq 3$,
 \begin{equation}\label{stat2} 
J=\left(\begin{array}{cc} 0 & I_3 \\ -I_3  & 0 \end{array}\right),\quad E=\left(\begin{array}{cc} \mathcal F& 0 \\0 & I_3\end{array}\right), \quad 
F(\bold w)=\left(\begin{array}{cc}  0\\  0   \\  0   \\  -\sin (u_1)   \\  -\sin (u_2)   \\  -\sin (u_3)  \end{array}\right)
\end{equation} 
 and where $I_3$ denotes the identity matrix of order $3$ and $ \mathcal F$ is the diagonal-matrix linear operator
 \begin{equation} 
 \label{Lapla}
 \mathcal{F}=\Big (\Big(-c_j^2\frac{d^2}{dx^2}
\Big)\delta_{j,k} \Big ),\quad\l1\leqq j, k\leqq 3.
   \end{equation}

For the $\mathcal{Y}$-junction of type I, we consider red the following family of self-adjoint  operators $\mathcal F_\lambda\equiv\mathcal F$ defined on the $\delta'$-interaction domain (see Proposition \ref{2M} in Appendix \S A)
\begin{equation}
\label{I3trav23}
D_\mathrm{I}(\mathcal{F}_\lambda) := \Big \{(v_j)_{j=1}^3\in H^2(\mathcal{Y}):
c_1v'_1(0-)=c_2v'_2(0+)=c_3v'_3(0+),\;\; \sum\limits_{j=2}^3 c_jv_j(0+)-c_1v_1(0-)=\lambda c_1v'_1(0-) \Big \},
\end{equation}
with $\lambda\in \mathbb R$.  We note from the later that the natural space to develop a local well-posedness theory for \eqref{stat1} is $H^1(\mathcal{Y}) \times L^2(\mathcal{Y})$ red such as will be performance in subsection 3.1.1 below).
%


\subsubsection{Stationary solutions on a $\mathcal{Y}$-junction of type I}
We are interested in the dynamics generated by the flow of the sine-Gordon model \eqref{sg2} around solutions of stationary type,
\[
u_j (x,t) = \phi_j(x), \qquad v_j(x,t) = 0,
\]
for all $j = 1,2,3$, and $x \in E_j$, $t > 0$, where each of the profile functions $\phi_j$ satisfies the equation
\begin{equation}
\label{trav21}
-c_j^2 \phi''_j + \sin \phi_j=0, 
\end{equation}
on each edge $E_j$ and for all $j$, as well as the boundary conditions in  \eqref{bcI} at the vertex $\nu = 0$, more precisely, 
\begin{equation}
\label{bcIfam}
c_1 \phi'_1(0-) = c_2 \phi'_2(0+) = c_3 \phi'_3(0+),\quad
-c_1  \phi_1(0-) + \sum_{j=2}^3 c_j  \phi_j(0+) = \lambda c_1 \phi'_1(0-),
\end{equation}
for some $\lambda \in \R$. Motivated by the well-known kink-type soliton profile solutions to the sine-Gordon equation on the full real line \cite{Dra83,SCM}, we consider initially the particular family of profiles having the form
\begin{equation}
\label{trav22}
\begin{cases}
\phi_1(x) = 4 \arctan \big( e^{(x-a_1)/c_1}\big), & x \in (-\infty,0),\\
\phi_j(x) = 4 \arctan \big( e^{-(x-a_j)/c_j}\big), & x \in (0,\infty), \,\; j=2,3,\\
\end{cases}
\end{equation}
where each $a_j$ is a constant determined by the boundary conditions \eqref{bcIfam} (see Figure \ref{figKsolitons} below). Notice as well that this family of stationary solutions \eqref{trav22} satisfies
\begin{equation} 
\label{bcinfty}
\phi_1(-\infty) = \phi_j(+\infty) = 0, \qquad j = 2,3
\end{equation}
(in other words, the constant of integration when solving \eqref{trav21} to arrive at \eqref{trav22} is zero on each edge $E_j$). This decaying behavior at $\pm \infty$, for instance, guarantees that $\Phi = (\phi_j)_{j=1}^3 \in H^2(\mathcal{Y})$.

Our second class of  solutions  to the sine-Gordon equation are the kink/anti-kink-type soliton, namely,
 profiles having the form (with $c_1=c_2=c_3=1$ without loss of generality)
\begin{equation}
\label{antikink}
\begin{cases}
\phi_1(x) = 4 \arctan \big( e^{-(x-a_1)}\big), & x \in (-\infty,0),\;\; \lim_{x\to -\infty}\phi_1(x)=2\pi,\\
\phi_j(x) = 4 \arctan \big( e^{-(x-a_j)}\big), & x \in (0,\infty),\;\;  \lim_{x\to +\infty}\phi_j(x)=0\;\; j=2,3,\\
\end{cases}
\end{equation}
where each $a_j$ is a constant determined by the boundary conditions \eqref{bcIfam}. Notice that $\phi_1$ is an anti-kink, with non-zero limit at $x = -\infty$, and hence, not belonging to $H^2(-\infty,0)$. It is coupled with two kinks at the other two edges, hence the name of kink/anti-kink structure (see Figure \ref{figAKsolitons} below).

In the forthcoming stability analysis, the family of linearized operators around the stationary profiles plays a fundamental role. These operators are characterized by the following self-adjoint diagonal matrix operators,
\begin{equation}\label{trav23}
\mathcal{L} \bold{v}=\Big (\Big(-c_j^2\frac{d^2}{dx^2}v_j + \cos (\phi_j)v_j
\Big)\delta_{j,k} \Big ),\quad \l1\leqq j, k\leqq 3,\;\;\bold{v}= (v_j)_{j=1}^3,
\end{equation}
where $\delta_{j,k}$ denotes the Kronecker symbol, and defined on domains with $\delta'$-type interaction at the vertex $\nu = 0$, $D(\mathcal{L}_\lambda) \equiv D_\mathrm{I}(\mathcal{F}_\lambda)$, for admissible parameters $\lambda$ that ensure the existence of profiles $(\phi_j)_{j=1}^3\in D_\mathrm{I}(\mathcal{F}_\lambda)$. 
 It is to be observed that the particular family \eqref{trav22} of kink-profile stationary solutions under consideration is such that $\Phi = (\phi_j)_{j=1}^3 \in D(\mathcal{L}_\lambda)$ in view that they satisfy the boundary conditions \eqref{bcIfam}.  An interesting characteristic of the spectrum structure associated with operators in \eqref{trav23}  on metric graphs is that they have a nontrivial Morse index (in general bigger or equal to 1) which makes the stability study not so immediate. Here we will use a novel  linear instability criterion  for stationary solutions of evolution models on metric graphs developed by Angulo\&Cavalcante in \cite{AngCav} (see also \cite{AngCav2}).

\subsubsection{Summary of results}

Let us summarize the main contributions of this paper and sketch the structure of the paper: 
\begin{itemize}
\item[$-$] First, in section \S \ref{seccriterium}, we review the general instability criterion for stationary solutions for the sine-Gordon model \eqref{sg2} on a $\mathcal{Y}$-junction developed in the companion paper \cite{AnPl-delta} (see Theorem \ref{crit} below. See also \cite{AngCav}). It essentially provides sufficient conditions on the flow of the semigroup generated by the linearization around the stationary solutions, for the existence of a pair of positive/negative real eigenvalues of the linearized operator based on its Morse index. It is to be observed that this instability criterion is very versatile, as it applies to any type of stationary solutions (such as anti-kinks or breathers, for example) and for different interactions at the vertex, such as both the $\delta$- and $\delta'$-types.
\item[$-$] The central section \S \ref{secinsI} is devoted to develop the instability theory of stationary solutions to the sine-Gordon equation with $\delta'$-interaction on a $\mathcal{Y}$-junction of type I. First, we focus on the kink-profile type waves defined in \eqref{trav22}-\eqref{bcinfty} and the  local well-posedness problem associated to \eqref{sg2} with a $\delta'$-interaction. In section \S \ref{secK} it is shown that, for a particular class of profiles satisfying an extra continuity condition (see \eqref{contcase} below) and for specific conditions on the parameters $a_j$ and $c_j$ (see $(i)$ and $(ii)$ in Theorem \ref{2main} below) the profiles are linearly and nonlinearly unstable. The result is based on a Morse index calculation and on the application of Theorem \ref{crit}. Related Morse index calculations for the remaining cases (but not yet conclusive in terms of stability and performed for later use)  can be found in Appendix \S \ref{secMorse}.  In section \S \ref{secAK} it is established the instability property of  the kink/anti-kink profiles type waves defined in \eqref{antikink} for also specific conditions on the parameters $a_j$ and $\lambda$ (see  Theorem \ref{0antimain} below).
\item[$-$] By convenience of the reader and by the sake of completeness we establish in Appendix \S A and \S B some results of the extension theory of symmetric operators used in the body of the manuscript.  Moreover, we also establish some Morse index calculations related to the kink/anti-kink profiles for a possible future
study.

\end{itemize}

\subsection*{On notation}

For any $-\infty\leq a<b\leq\infty$, we denote by $L^2(a,b)$ the Hilbert space equipped with the inner product $
(u,v)=\int_a^b u(x)\overline{v(x)}dx$.
By $H^n(a,b)$  we denote the classical  Sobolev spaces on $(a,b)\subseteq \mathbb R$ with the usual norm.   We denote by $\mathcal{Y}$ the junction of type I parametrized by the edges $E_1 = (-\infty,0)$, $E_j = (0,\infty)$, $j =2,3$,  attached to a common vertex $\nu=0$. On the graph $\mathcal{Y}$ we define the classical spaces 
  \begin{equation*}
  L^p(\mathcal{Y})=L^p(-\infty, 0)  \oplus L^p(0, +\infty) \oplus L^p(0, +\infty), \quad \,p>1,
  \end{equation*}   
 and 
  \begin{equation*}  
 \quad H^m(\mathcal{Y})=H^m(-\infty, 0) \oplus H^m(0, +\infty)  \oplus H^m(0, +\infty), 
 \end{equation*}   
with the natural norms. Also, for $\mathbf{u}= (u_j)_{j=1}^3$, $\mathbf{v}= (v_j)_{j=1}^3 \in L^2(\mathcal{Y})$, the inner product is defined by
$$
\langle \mathbf{u}, \mathbf{v}\rangle= \int_{-\infty}^0 u_1(x) \overline{v_1(x)} \, dx  + \sum_{j=2}^3 \int_0^{\infty}  u_j(x) \overline{v_j(x)} \, dx
$$
Depending on the context we will use the following notations for different objects. By $\|\cdot \|$ we denote  the norm in $L^2(\mathbb{R})$ or in $L^2(\mathcal{Y})$. By $\| \cdot\| _p$ we denote  the norm in $L^p(\mathbb{R})$ or in $L^p(\mathcal{Y})$.   Finally, if $A$ is a closed, densely defined symmetric operator in a Hilbert space $H$ then its domain is denoted by $D(A)$, the deficiency indices of $A$ are denoted by  $n_\pm(A):=\dim  \ker (A^*\mp iI)$, where $A^*$ is the adjoint operator of $A$, and the number of negative eigenvalues counting multiplicities (or Morse index) of $A$ is denoted by  $n(A)$. 

 \section{Preliminaries: Linear instability criterion for sine-Gordon model on a $\mathcal Y$-junction}
\label{seccriterium}

In this section we review the linear instability criterion of stationary solutions for the sine-Gordon model \eqref{sg2} on a  $\mathcal{Y}$-junction developed in \cite{AnPl-delta} (see also \cite{AngCav, AngCav2}). Although the stability analysis in \cite{AnPl-delta} pertains to interactions of $\delta$-type at the vertex, it is important to note that the criterion proved in that reference also applies to any type of stationary solutions independently of the boundary conditions under consideration and, therefore, it can be used to study the present configurations with boundary rules at the vertex of $\delta'$-interaction type, or even to other types of stationary solutions to the sine-Gordon equation such as breathers, for instance. In addition, the criterion applies to both the $\mathcal Y$-junction of type I (see Figure \ref{figYtipoI}) and of type II (see Figure \ref{figYtipoII}). 

Let $\mathcal{Y}$ be a $\mathcal{Y}$-junction of type I or II. Let us suppose that $JE$ on a domain $D(JE)\subset H^1(\mathcal Y)\times L^2(\mathcal Y)$ is the infinitesimal generator of a $C_0$-semigroup on $H^1(\mathcal Y)\times L^2(\mathcal Y)$ and that there exists an stationary solution $\Upsilon=(\zeta_1, \zeta_2, \zeta_3,0,0,0)\in D(JE)$. Thus, every component $\zeta_j$ satisfies the equation
\begin{equation}\label{statio}
-c^2_j \zeta''_j + \sin (\zeta_j)=0,\quad j=1,2,3.
\end{equation}
 Now, we suppose  that $\bold w$ satisfies formally  equality in \eqref{stat1} and we define
  \begin{equation}\label{stat3}
 \bold v \equiv \bold w-\Upsilon,
  \end{equation}
  then, from \eqref{statio} we obtain the following  linearized system for \eqref{stat1} around $\Upsilon$,
   \begin{equation}\label{stat4}
  \bold v_t=J\mathcal E\bold v,
 \end{equation} 
  with $\mathcal E$ being the $6\times 6$ diagonal-matrix $\mathcal E=\left(\begin{array}{cc} \mathcal L & 0 \\0 & I_3\end{array}\right)$, and 
 \begin{equation}\label{stat5}
\mathcal{L} =\Big (\Big(-c_j^2\frac{d^2}{dx^2}+ \cos (\zeta_j)
\Big)\delta_{j,k} \Big ),\qquad 1\leqq j, k\leqq 3.
\end{equation}
We point out the equality $J\mathcal E=J E+\mathcal{T}$, with 
$$
\mathcal{T} = \left(\begin{array}{cc} 0& 0 \\ \big( - \cos(\zeta_j) \, \delta_{j,k} \big)& 0\end{array}\right)
$$
 being a \emph{bounded} operator on $H^1(\mathcal Y)\times L^2(\mathcal Y)$. This  implies that $J\mathcal E$ also generates a  $C_0$-semigroup on  $H^1(\mathcal Y)\times L^2(\mathcal Y)$ (see Pazy \cite{Pa}).
  
The linear instability criterion provides sufficient conditions for the trivial solution $\bold v \equiv 0$ to be unstable by the linear flow of \eqref{stat4}. More precisely, it underlies the existence of a {\it growing mode solution} to \eqref{stat4} of the form $ \bold v=e^{\lambda t} \Psi$ and $\RE \lambda >0$. To find it, one needs to solve the formal system 
 \begin{equation}\label{stat6}
 J\mathcal E \Psi=\lambda \Psi,
\end{equation}
with $\Psi\in D(J\mathcal E)$. If we denote by $\sigma(J\mathcal E)= \sigma_{\mathrm{pt}}(J\mathcal E)\cup \sigma_{\mathrm{ess}}(J\mathcal E)$ the spectrum  of $J\mathcal E$ (namely, $\lambda \in  \sigma_{\mathrm{pt}}(J\mathcal E)$ if $\lambda$ is isolated and with finite multiplicity) then we have the following
\begin{definition}
The stationary vector solution $\Upsilon \in D(\mathcal E)$    is said to be \textit{spectrally stable} for the sine-Gordon model \eqref{stat1} if the spectrum of $J\mathcal E$, $\sigma(J\mathcal \mathcal E)$, satisfies $\sigma(J\mathcal E)\subset i\mathbb{R}.$
Otherwise, the stationary solution $\Upsilon\in D(\mathcal E)$   is said to be \textit{spectrally unstable}.
\end{definition}
\begin{remark}
It is well-known that $ \sigma_{\mathrm{pt}}(J\mathcal E)$ is symmetric with respect to both the real and imaginary axes and $ \sigma_{\mathrm{ess}}(J\mathcal E)\subset i\mathbb{R}$ under the assumption that $J$ is skew-symmetric and that $\mathcal E$ is self-adjoint (by supposing, for instance, Assumption $(S_3)$ below for $\mathcal L$; see \cite[Lemma 5.6 and Theorem 5.8]{GrilSha90}). These cases on $J$ and  $\mathcal E$ are considered in the theory. Hence, it is equivalent to say that $\Upsilon\in D(J\mathcal E)$ is  \textit{spectrally stable} if $ \sigma_{\mathrm{pt}}(J\mathcal E)\subset i\mathbb{R}$, and it is spectrally unstable if $ \sigma_{\mathrm{pt}}(J\mathcal E)$ contains point $\lambda$ with  $\RE \lambda>0.$
\end{remark}

 It is widely known  that the spectral instability of a specific traveling wave solution of an evolution type model is   a key prerequisite to show their nonlinear instability property (see \cite{GrilSha90, Lopes, ShaStr00} and references therein). Thus we have the following definition.

 \begin{definition}\label{nonstab}
The stationary vector solution $\Upsilon \in D(\mathcal E)$    is said to be \textit{nonlinearly unstable} in $X\equiv H^1(\mathcal Y)\times L^2(\mathcal Y)$-norm for model sine-Gordon \eqref{stat1} if there is $\epsilon>0$ such that for every $\delta>0$ there exist an initial data $\bold w_0$ with $\|\Upsilon -\bold w_0\|_X<\delta$ and an instant $t_0=t_0(\bold w_0)$, such that $\|\bold w(t_0)-\Upsilon \|_X>\epsilon$, where $\bold w=\bold w(t)$ is the solution of the sine-Gordon model with initial data $\bold w(0)=\bold w_0$.
\end{definition}
 
 From \eqref{stat6}, our eigenvalue problem to solve is now reduced to,
  \begin{equation}\label{stat11}    
 J\mathcal E\Psi =\lambda \Psi, \quad \RE \lambda >0,\;\;\Psi\in D(\mathcal E).
 \end{equation}  
 Next, we establish our theoretical framework and assumptions for obtaining a nontrivial solution to  problem in \eqref{stat11}:
 \begin{enumerate}
 \item[($S_1$)]  $J\mathcal E$ is the generator of a $C_0$-semigroup $\{S(t)\}_{t\geqq 0}$.  
 \item[($S_2$)] Let $\mathcal L$ be the matrix-operator in \eqref{stat5}  defined on a domain $D(\mathcal L)\subset L^2(\mathcal Y)$ on which $\mathcal L$ is self-adjoint.
 \item[($S_3$)] Suppose $\mathcal L:D(\mathcal L)\to L^2(\mathcal Y)$ is  invertible  with Morse index $n(\mathcal L)=1$ and such that $\sigma(\mathcal L)=\{\lambda_0\}\cup J_0$ with $J_0\subset [r_0, +\infty)$, for $r_0>0$, and $\lambda_0<0$,
 \end{enumerate} 
 
The criterion for linear instability reads precisely as follows.
\begin{theorem}[linear instability criterion \cite{AnPl-delta, AngCav, AngCav2}]
\label{crit}
Suppose the assumptions $(S_1)$ - $(S_3)$ hold.  Then the operator $J\mathcal E$ has a real positive and a real negative eigenvalue. 
\end{theorem}
\begin{proof}
See the proof of Theorem 3.2 in \cite{AnPl-delta} (see also  \cite{AngCav}).
\end{proof}
\begin{remark}
The proof of Theorem \ref{crit} is based on
the characterization of eigenvectors associated to a nonnegative eigenvalue for bounded linear operators with invariant closed convex cones (see Theorem 3.3 in \cite{AnPl-delta}, or Krasnoselskii \cite{Kra}, Chapter 2, section 2.2.6).  We call the attention that the instability framework developed in Grillakis\&Shatah\&Strauss in \cite{GrilSha90} can not be applied in our study due to loss of the  translation symmetry property of our model on $\mathcal Y$-junctions.
\end{remark}

\section{Instability of stationary solutions for the sine-Gordon equation with $\delta'$-interaction on a $\mathcal{Y}$-junction of type I}
\label{secinsI}

In this section we study the stability of stationary solutions determined by a $\delta'$-interaction type at the vertex  $\nu=0$ of a $\mathcal{Y}$-junction of type I. First we study the kink-profile type in \eqref{trav22}-\eqref{bcinfty} and the the local well-posedness problem associated to \eqref{sg2} with a $\delta'$-interaction. Next, we examine the structure of the stationary wave solutions under consideration. Finally, we apply the linear instability criterion (Theorem \ref{crit}) to prove that the family of stationary solutions of kink type \eqref{trav22} are linearly (and nonlinearly) unstable (see Theorem \ref{2main} below). Our second goal is the study of the kink/anti-kink type of profile defined in \eqref{antikink} and, similarly as in the former kink-type case, we establish the necessary ingredients for obtaining our instability results.

\subsection{Kink-profile instability  on a $\mathcal{Y}$-junction of type I}
\label{secK}

Let us start our stability study for the kink-profile type in \eqref{trav22}. First, we examine the Cauchy problem associated to sine-Gordon model in \eqref{stat1}. As this study is not completely  standard in the case of metric graphs we focus on the new ingredients that arise. 

\subsubsection{The Cauchy problem}
\label{seclocalWP}

We establish the local well-posedness of vectorial equation \eqref{stat1} in $H^1(\mathcal Y)\times L^2(\mathcal Y)$ with a $\mathcal Y$-junction of first type. Since the proof is similar to its $\delta$-interaction counterpart (see Section \S 2 in \cite{AnPl-delta}), we gloss over many details and specialize the discussion to the points that are particular to the $\delta'$-interaction case. 

We start by studying  operator $\mathcal F$ in \eqref{Lapla} (which will be denoted here by $\mathcal F=\mathcal H_\lambda$) and defined on the $\delta' $-interaction domain $D(\mathcal H_\lambda)$:
\begin{equation}\label{3trav23}
D(\mathcal{H}_{\lambda})= \Big \{(v_j)_{j=1}^3\in H^2(\mathcal{Y}):
c_1v'_1(0-)=c_2v'_2(0+)=c_3v'_3(0+),\;\; \sum_{j=2}^3 c_jv_j(0+)-c_1v_1(0-)=\lambda c_1v'_1(0-) \Big \},
\end{equation}
with $\lambda\in \mathbb R$. In the sequel we  establish the spectral properties of the family of self-adjoint operator $(\mathcal H_\lambda, D(\mathcal H_{\lambda}))$ (see Proposition \ref{2M} in  Appendix \S A).

\begin{theorem}\label{spectrum2}
Let $\lambda\in \mathbb R-\{0\}$. Then  the essential spectrum of $(\mathcal H_\lambda, D(\mathcal H_{ \lambda}))$ is purely absolutely continuous and $\sigma_{\mathrm{ess}}(\mathcal H_\lambda)=\sigma_{\mathrm{ac}}(\mathcal H_\lambda)=[0,+\infty)$. If $ \lambda<0$ then $\mathcal H_\lambda$ has precisely one negative, simple eigenvalue, {\it i.e.}, its point spectrum $\sigma_{\mathrm{pt}}(\mathcal H_\lambda)$ is 
$$
\sigma_{\mathrm{pt}}(\mathcal H_\lambda)=\Big \{-\frac{1}{\lambda^2}(\sum_{j=1}^3|c_j|)^2\Big\},
$$ 
such that for $\alpha= \frac{1}{\lambda}\sum_{j=1}^3|c_j|$, $\Phi_\lambda=(-\sign(c_1)e^{-\frac{\alpha}{|c_1|} x}, \sign(c_2)e^{\frac{\alpha}{|c_2|} x}, \sign(c_3)e^{\frac{\alpha}{|c_3|} x})$ is the associated eigenfunction. If $ \lambda> 0$ then $\mathcal H_\lambda$ has no eigenvalues, $\sigma_{\mathrm{pt}}(\mathcal H_\lambda)=\varnothing$.
\end{theorem}

\begin{proof} By applying Proposition \ref{2M} in Appendix \ref{secAppET} and by following the same ideas 
as in the proof of Theorem 2.1 in \cite{AnPl-delta}, we obtain the conclusion with respect to the Morse index for $\mathcal H_\lambda$. Moreover, since the operator $A=-\frac{d^2}{dx^2}$ with the Neumann-domain $D_{\mathrm{Neu}}=\{f\in H^2(0,+\infty): f'(0+)=0\}$ has $\sigma_{\mathrm{ess}}(A)=[0,+\infty)$, we obtain the statement about the essential spectrum of $\mathcal F_\lambda$. This finishes the proof.
\end{proof}

\begin{theorem}\label{cauchy2}
Let $\lambda\in \mathbb R-\{0\}$ and consider the linear operators $J$ and $E$ defined in \eqref{stat2}. Then, $\mathcal A\equiv JE$ with $D(\mathcal A)= D(\mathcal H_\lambda)\times H^1(\mathcal Y)$ is the infinitesimal  generator of a $C_0$-semigroup $\{G(t)\}_{t\geqq 0}$ on $H^1(\mathcal Y)\times L^2(\mathcal Y)$. The initial value problem 
\begin{equation}\label{LW1}
\begin{cases}
\bold z_{t}=\mathcal A\bold z \\
\bold z(0)=\bold z_0\in D(\mathcal A)=D(\mathcal H_\lambda)\times H^1(\mathcal Y)
\end{cases}
\end{equation}
has a unique solution $\bold z\in C([0, +\infty): D(\mathcal A))\cap C^1((0, +\infty): H^1(\mathcal Y)\times L^2(\mathcal Y))$ given by $\bold z(t)=G(t)\bold z_0$, $t\geqq 0$. Moreover, for any $\Psi\in H^1(\mathcal Y)\times L^2(\mathcal Y)$ and $\theta>\beta_0+1$ we have the representation formula
\begin{equation}\label{FRW}
G(t)\Psi=\frac{1}{2\pi i}\int_{\theta-i\infty}^{\theta+i\infty} e^{\eta t} R(\eta: \mathcal A) \Psi d\eta
\end{equation}
where $\eta\in \rho(\mathcal A)$ with $\RE \eta = \theta$ and $R(\eta: \mathcal A)=(\eta I-\mathcal A)^{-1}$, and for every $\delta>0$, the integral converges uniformly in $t$ for every $t\in [\delta, 1/\delta]$. Here $\beta_0=\frac{1}{\lambda^2} (\sum_{j=1}^3|c_j|)^2$.
\end{theorem}

\begin{proof} The proof follows the same strategy as in Theorem 2.5 in \cite{AnPl-delta}: apply Theorem \ref{spectrum2}, without loss of generality set $c_j^2=1$, and consider the following inner product in $H^1(\mathcal Y)$, 
\begin{equation}\label{2inner}
\begin{aligned}
\langle \bold u, \bold v\rangle_{1,\lambda}=& \int_{-\infty}^0 u'_1\overline{v'_1}dx +\sum_{j=2}^3 \int_0^{\infty} u'_j\overline{v_j}dx +(\beta_0 +1) \langle \bold u, \bold v\rangle \\
&+ \frac{1}{\lambda}\Big[\sum_{j=2}^3 c_ju_j(0+)-c_1u_1(0-)\Big]\Big[\sum_{j=2}^3 c_j\overline{v_j}(0+)-c_1\overline{v_1}(0-)\Big],
\end{aligned}
\end{equation}
which induces a norm, $\|\cdot\|_{1,\lambda}$, equivalent to the standard norm in $H^1(\mathcal Y)$ and given by
\begin{equation}\label{2norm}
\|\bold v\|_{1,\lambda}^2=\|\bold v'\|^2 _{L^2(\mathcal Y)} + (\beta_0+1)\|\bold v\|^2 _{L^2(\mathcal Y)} +\frac{1}{\lambda}\Big|\sum_{j=2}^3 c_ju_j(0+)-c_1u_1(0-)\Big|^2,
\end{equation}
where for $\lambda<0$, $\beta_0=\frac{9}{\lambda^2}$, and for $\lambda> 0$, $\beta_0=0$. Similar arguments to those in the proof of Theorem 2.5 in \cite{AnPl-delta} yield the conclusion.
\end{proof}

Lastly, by using the contraction mapping principle as in the proof of Theorem 2.7 in \cite{AnPl-delta}, we obtain the following local well-posedness theorem for the sine-Gordon equation on $H^1(\mathcal{Y})\times L^2(\mathcal Y) $. The proof is omitted.

\begin{theorem}\label{well2} 
	For any $\Psi \in H^1(\mathcal{Y})\times L^2(\mathcal Y) $
	there exists $T > 0$ such that the sine-Gordon
	equation \eqref{stat1} has a unique solution $\mathbf w \in C
	([0,T]; H^1(\mathcal{Y})\times  L^2(\mathcal Y) )$ satisfying  $\mathbf
	w(0)=\Psi$. For
	each $T_0\in (0, T)$ the mapping data-solution
	\begin{equation}\label{mapping}
	\Psi\in H^1(\mathcal{Y})\times  L^2(\mathcal Y) \to \mathbf w \in C ([0,T_0];
	H^1(\mathcal{Y})\times  L^2(\mathcal Y) ),
	\end{equation}
	 is at least of class $C^2$.

\end{theorem}

\subsubsection{The kink solution with  specific profile on a $\mathcal Y$-junction }
\label{secprofilesI}

We will consider  a specific class of stationary profiles $\Psi_{\lambda, \delta'}$ for the  sine-Gordon equation \eqref{sg2} on a $\mathcal Y$-junction  with profiles determined by  formulae \eqref{trav22} and satisfying the boundary conditions \eqref{bcIfam} of $\delta'$-type and intensity $\lambda \in \R$ at the vertex. In this fashion, we reckon that they belong to the $\delta'$-interaction type domain in \eqref{3trav23} with $c_j>0$ and $\lambda\in \mathbb R$. Here we will consider the {\it continuity case}, namely, 
\begin{equation}
\label{contcase}
\phi_2(0+)= \phi_3(0+).
\end{equation}
Consequently we obtain the condition $a_2/c_2=a_3/c_3$. Now, the conditions \eqref{bcIfam} for the family yield
\begin{equation}\label{deriva3}
\cosh(a_1/c_1)=\cosh(a_2/c_2).
\end{equation}
Thus, we obtain $a_1/c_1=\pm a_2/c_2$. Moreover, for $y=e^{a_2/c_2}$ and $y_1=e^{a_1/c_1}$ there holds
\[
(c_2+c_3)\arctan(y) +c_1\arctan(1/y_1) =-\lambda \frac{y_1}{1+ y_1^2}.
\]

For concreteness, in what follows we study the case $a_1/c_1= - a_2/c_2$ (see Remark \ref{cj} below on the remaining cases). Thus, we obtain the relation
\begin{equation}\label{deriva4}
\arctan(y)\sum_{j=1}^3c_j=-\lambda \frac{y}{1+y^2},\quad y> 0.
\end{equation}
Therefore, we conclude that, necessarily, $\lambda\in \Big(-\infty, -\sum_{j=1}^3c_j \Big)$.
Moreover, from the strictly-increasing property of the positive function, $y \mapsto \frac{1+ y^2}{y}\arctan(y)$, $y>0$, we obtain from \eqref{deriva4} the existence of a smooth shift-map (also real analytic), $
\lambda\in (-\infty, -\sum_{j=1}^3c_j)\mapsto a_2(\lambda)$
satisfying \eqref{deriva4}, and such that the mapping 
$$
\lambda\in \Big(-\infty, -\sum_{j=1}^3c_j\Big)\mapsto \Psi_{\lambda, \delta'}=(-\phi_{1,a_1(\lambda)}, \phi_{2, a_2(\lambda)}, \phi_{3, a_3(\lambda)},0,0,0),
$$
 represents a real-analytic family of static profiles for the sine-Gordon equation on a $\mathcal Y$-junction of first-type satisfying  $\delta'$-interaction type  at the vertex $\nu=0$.

Hence we obtain, for $a_i=a_i(\lambda)$ and $\phi_i=\phi_{i,a_i(\lambda)}$, the following behavior:
\begin{enumerate}
\item[1)] for  $\lambda\in (-\infty, -\frac{\pi}{2} \sum_{j=1}^3 c_j)$ we obtain $a_2>0$:  therefore $a_3>0$, $a_1<0$,  $ \phi'_i<0$ and $\phi_i''(a_i) =0$, for $i=1,2,3$. Moreover, $\phi_i\in (0, \eta)$, $i=2, 3$, $-\phi_1\in (-\eta, 0)$, with $\eta=4 \arctan \big(e^{a_2/c_2}\big)>\pi$. Thus, the profile of $(-\phi_1, \phi_2,\phi_3)$, looks similar to the one shown in Figure \ref{figBumpYI} below (bump-type profile); 
\item[2)] for   $\lambda\in (-\frac{\pi}{2} \sum_{j=1}^3 c_j, -\sum_{j=1}^3 c_j)$ we obtain $a_2<0$:  therefore $a_3<0$, $a_1>0$,   $ \phi'_i<0$ and $ \phi''_i>0$ for $i=2,3$,  $\phi'_1>0$ and $\phi''_1>0$. $\phi_j\in (0, \pi)$ for every $j$. Thus, the profile of $(-\phi_1, \phi_2,\phi_3)$ looks similar to that in Figure \ref{figTailYI} below (tail-type profile); 
\item[3)]  the case  $\lambda=-\frac{\pi}{2} \sum_{j=1}^3 c_j$ implies $a_1=a_2=a_3=0$: Therefore,  $\phi_1(0)=\phi_2(0) =\phi_3(0)=\pi$. Moreover, 
$\phi_i''(0) =0$, $i=1,2, 3$. Thus, the profile of $(-\phi_1, \phi_2,\phi_3)$ looks similar to that in Figure \ref{figZeroWYI}.
\end{enumerate}

\begin{remark}
\label{cj} 
It is to be observed that we have left open the description of other kink-soliton profiles not satisfying the continuity property \eqref{contcase} at zero for the components $\phi_2, \phi_3$, as well as the case where $a_1/c_1=a_2/c_2$. These other profiles, however, can be studied following the spectral methods described here.
\end{remark}

\begin{figure}[t]
\begin{center}
\subfigure[$\lambda\in (-\infty, -\frac{\pi}{2} \sum_{j=1}^3 c_j)$]{\label{figBumpYI}\includegraphics[scale=.4, clip=true]{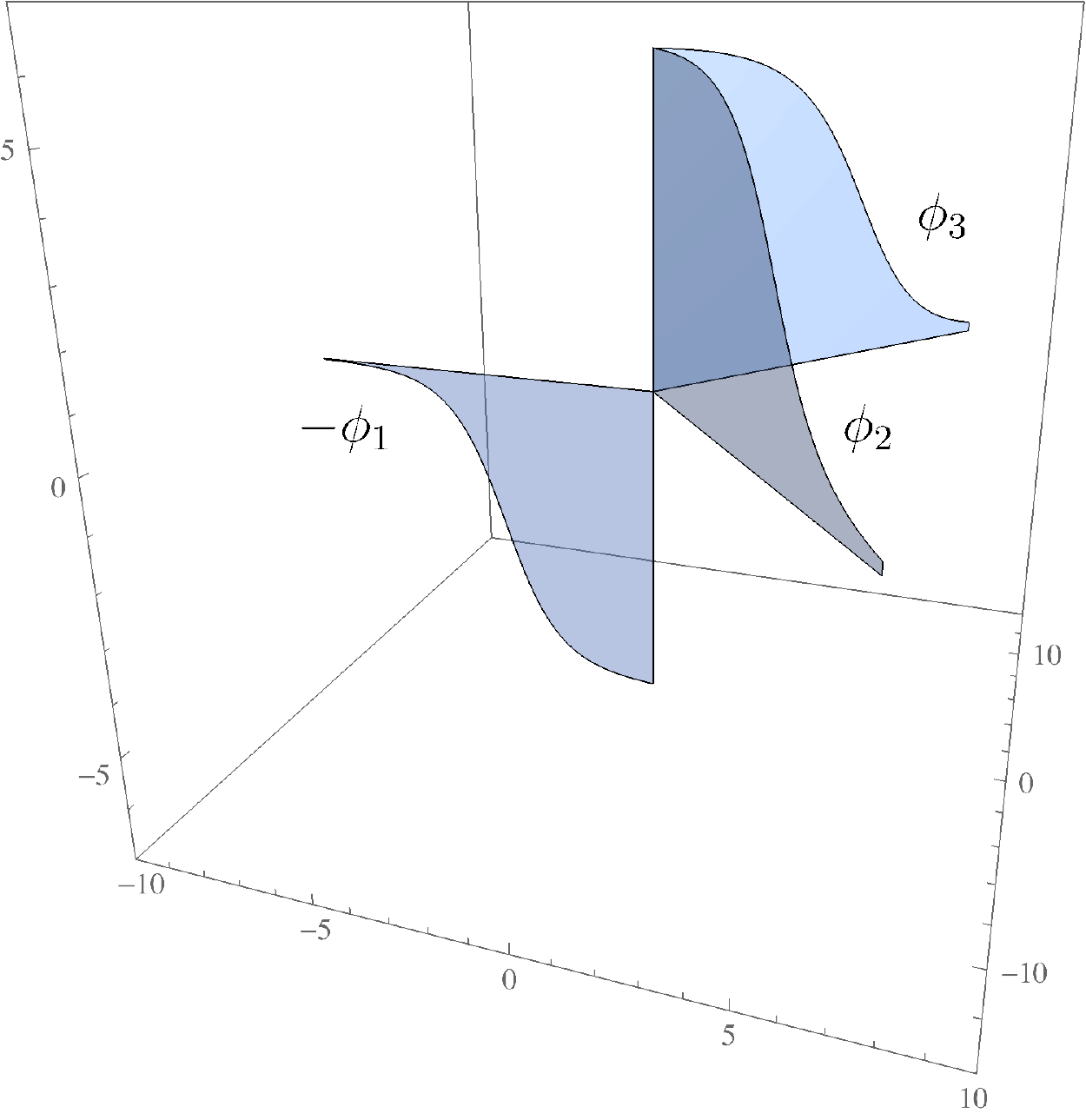}}
\subfigure[$\lambda\in (-\frac{\pi}{2} \sum_{j=1}^3 c_j, -\sum_{j=1}^3 c_j)$]{\label{figTailYI}\includegraphics[scale=.4, clip=true]{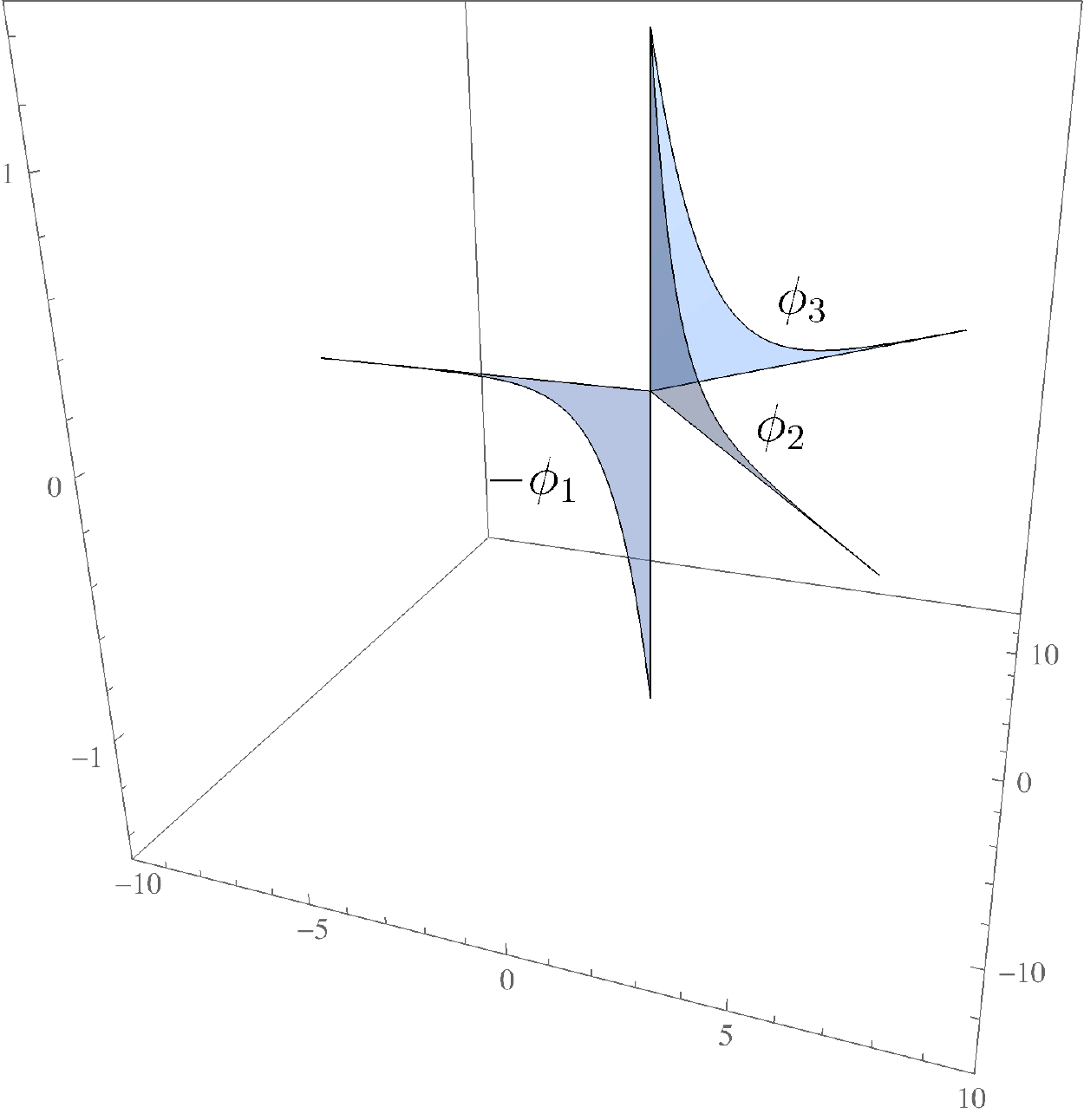}}
\subfigure[$\lambda=-\frac{\pi}{2} \sum_{j=1}^3  c_j $]{\label{figZeroWYI}\includegraphics[scale=.4, clip=true]{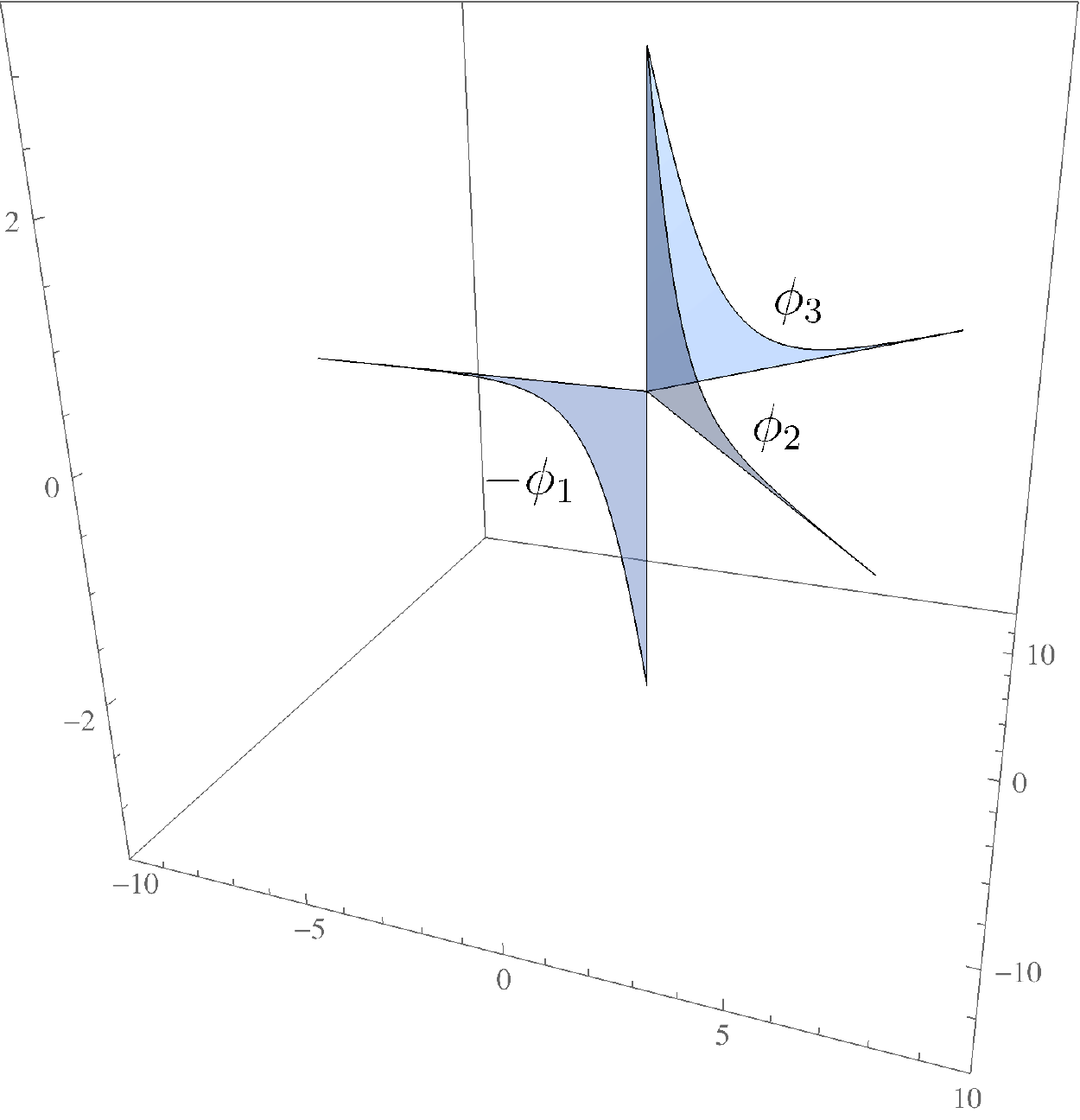}}
\end{center}
\caption{\small{Plots of stationary solutions $(-\phi_1, \phi_2, \phi_3)$ defined in \eqref{trav22} in the case where $c_j = 1$ for all $j=1,2,3$, for different values of $\lambda \in (\infty, -\sum_j c_j) = (-\infty, -3)$. Panel (a) shows the stationary profile solutions (``bump-type'' configuration) for the case $\lambda \in (-\infty,-3\pi/2)$. Panel (b) shows the profile of\/ ``tail-type'' for the case $\lambda \in (-3\pi/2,-3)$. Panel (c) shows the profile solutions when $\lambda = -3\pi/2$ (color online).}}\label{figKsolitons}
\end{figure}

Our instability result for the stationary profiles
$\Psi_{\lambda,\delta'}=(-\phi_1, \phi_2, \phi_3, 0,0,0)$ (with a slight abuse of notation) with $\phi_i=\phi_{i, a_i(\lambda)}$ defined in \eqref{trav22}-\eqref{deriva4} via $a_2(\lambda)$ and such that for $a_i=a_i(\lambda)$,
\begin{equation}
\label{ajs}
a_3 = \frac{c_3}{c_2} a_2,\;\;\text{and }\;\; 
a_1 = -\frac{c_1}{c_2} a_2,\quad c_i>0,
\end{equation}
is the following
\begin{theorem}
\label{2main} 
Let $\lambda\in  (-\infty, -\sum_{j=1}^3 c_j)$, $c_j>0$, and the smooth family of stationary profiles $\lambda\to \Psi_{\lambda,\delta'}$ determined above. Then $\Psi_{\lambda, \delta'}$ is spectrally and nonlinearly unstable for the sine-Gordon model \eqref{sg2} on a $\mathcal Y$-junction  of first type in the following cases:
\begin{enumerate}
\item[(i)] for $\lambda\in  (-\frac{\pi}{2} \sum_{j=1}^3 c_j, -\sum_{j=1}^3 c_j)$ and the constants $a_i$ and $c_i$ satisfying \eqref{ajs},
\item[(ii)]   for $\lambda\in (-\infty,-\frac{\pi}{2} \sum_{j=1}^3 c_j]$ and $c_1=c_2=c_3$ in \eqref{ajs}.
\end{enumerate}
\end{theorem}

The proof of Theorem \ref{2main} is splitted into two parts, pertaining to the cases $(i)$ and $(ii)$ described in its statement. The two cases are treated with different methodological approaches: case $(i)$ falls into the framework of extension theory, whereas case $(ii)$ requires an application of analytic perturbation theory.

\subsubsection{The spectral study in the case of $\lambda\in [-\frac{\pi}{2} \sum_{j=1}^3 c_j, -  \sum_{j=1}^3 c_j)$}

In this subsection we provide the spectral information about the family of self-adjoint operators $(\mathcal{L}_\lambda, D(\mathcal{L}_\lambda))$ where
\begin{equation}\label{deriva1}
\mathcal{L}_\lambda =\Big (\Big(-c_j^2\frac{d^2}{dx^2}+\cos(\phi_j)
\Big)\delta_{j,k} \Big ),\quad \l1\leqq j, k\leqq 3,
\end{equation}
associated to the linearization around the solutions $(\phi_j)_{j=1}$ determined in the previous subsection. Here $D(\mathcal{L}_\lambda)$ is the $\delta'$-interaction domain defined  in \eqref{3trav23} (see also Proposition \ref{2M} at Appendix \S A).

We begin by proving a result that applies to all values of $\lambda$ under consideration.
\begin{proposition}\label{3main}
Let $\lambda\in  (-\infty, -\sum_{j=1}^3 c_j)$, $\lambda\neq  -\frac{\pi}{2} \sum_{j=1}^3 c_j$. Then, $\ker( \mathcal{L}_\lambda )=\{\mathbf{0}\}$. Moreover, $\sigma_{\mathrm{ess}}(\mathcal{L}_\lambda)=[1,+\infty)$.
\end{proposition}

\begin{proof} Let $\bold{u}=(u_1, u_2, u_3)\in D(\mathcal{L}_\lambda)$ and $\mathcal{L}_\lambda \bold{u}=\bold{0}$. Then, from Sturm-Liouville theory on half-lines (see \cite{BeSh}) one obtains
\begin{equation}\label{2spec}
u_1(x)=\alpha_1\phi'_1(x), \;\;x<0,\quad u_j(x) = \alpha_j\phi'_j(x), \;\;x>0,\;\; j=2,3,
\end{equation}
for some $\alpha_1$ and $\alpha_j$, $j = 2,3$. Since $a_1/c_1 = - a_2/c_2 = - a_3/c_3$ we obtain $\phi_1(0-)=\phi_2(0+)=\phi_3(0+)$, and so $c_2^2 \phi''_2(0+)=c_3^2 \phi''_3(0+)$. Therefore, from the conditions on $\phi'_j(0)$ we deduce $
c_2\alpha_2  \phi''_2(0+)=c_3\alpha_3  \phi''_3(0+)$, 
and so $c_3\alpha_2  \phi''_3(0+)=c_2\alpha_3  \phi''_3(0+)$. 
Since $\phi''_3(0+)\neq 0$ we get $\alpha_2/c_2 = \alpha_3/c_3$. Similarly, we have $\alpha_2/c_2 = \alpha_1/c_1$. Next, the jump condition  implies 
\begin{equation}\label{-2spec}
\phi'_2(0)\sum_{j=1}^3 \alpha_j =\phi'_2(0)\frac{\alpha_1}{c_1}\sum_{j=1}^3 c_j =
\lambda  \alpha_2 \phi''_2(0).
\end{equation}
Now, we suppose $\alpha_2\neq 0$. Thus,
\begin{equation}\label{3spec}
\phi'_2(0)\frac{1}{c_2}\sum_{j=1}^3 c_j =
\lambda   \phi''_2(0).
\end{equation}
 Let us to consider the following cases:
\begin{enumerate}
\item[1)] Let $\lambda\in (-\infty,  -\frac{\pi}{2} \sum_{j=1}^3 c_j)$.  Then, the profile of $\phi_2$ satisfies  $\phi_2'(0)<0$ and it is of bump-type and so $\phi_2''(0)<0$. Therefore, from \eqref{3spec} we get a contradiction.
\item[2)]  Let $\lambda\in (-\frac{\pi}{2} \sum_{j=1}^3 c_j, -\sum_{j=1}^3 c_j)$. In this case, $\phi_2$ has a tail-type profile. Next, from the explicit formula for $\phi_2$ in \eqref{trav22}, \eqref{3spec} and \eqref{deriva4} we get
\begin{equation}\label{4spec}
(1-y^2) \arctan(y)=y,\qquad y=e^{a_2/c_2}\in (0,1).
\end{equation}
We arrive at a contradiction, in view that $h(x)=(1-x^2)\arctan(x)-x$ is a negative strictly decreasing mapping on $(0,1)$. 
\end{enumerate}
Thus, from the two cases above we need to have $\alpha_2=0=\alpha_3=\alpha_1$. The statement $\sigma_{\mathrm{ess}}(\mathcal{L}_\lambda)=[1,+\infty)$ is an immediate consequence of Weyl's Theorem (cf. \cite{RS4}). This finishes the proof.
\end{proof}

\begin{proposition}\label{4main}
Let $\lambda_0=  -\frac{\pi}{2} \sum_{j=1}^3 c_j$. Then, $\dim(\ker( \mathcal{L}_{\lambda_0}))=2$.
\end{proposition}

\begin{proof}
From \eqref{-2spec} we have $\alpha_1 + \alpha_2 +\alpha_3=0$ because of $\phi_2''(0+)=0$ and $\phi_2'(0+)\neq 0$. Then $\Phi_1=(-\phi'_1, \phi_2',0)$ and $\Phi_2=(0, \phi_2',-\phi'_3)$ belong to $D( \mathcal{F}_{\lambda_0})$ and $\Span\{\Phi_1, \Phi_2\}=\ker( \mathcal{L}_{\lambda_0})$.
\end{proof}

\begin{remark}\label{orthog} For $\Phi_1, \Phi_2$ in the proof of Proposition \ref{4main} we have the orthogonality relations $(\phi_1,\phi_2,\phi_3) \bot \Phi_1$ and   $(\phi_1,\phi_2,\phi_3)\bot \Phi_2$. Therefore, $(\phi_1,\phi_2,\phi_3)\in [\ker( \mathcal{L}_{\lambda_0})]^\bot$.
\end{remark}

\begin{proposition}\label{5main}
Let $\lambda\in [-\frac{\pi}{2} \sum_{j=1}^3 c_j, -\sum_{j=1}^3 c_j)$. Then $n( \mathcal{L}_\lambda )=1$.
\end{proposition}

\begin{proof}  From Proposition \ref{2M} in Appendix \S A we have that the family $(\mathcal{L}_\lambda, D(\mathcal{L}_\lambda))$ represents all the self-adjoint extensions of the closed symmetric operator, $(\mathcal H_0, D(\mathcal H_0))$, where
\begin{equation}\label{6spec}
\mathcal{H}_0=\Big (\Big(-c_j^2\frac{d^2}{dx^2}+\cos(\phi_j)
\Big)\delta_{j,k} \Big ),\;\l1\leqq j, k\leqq 3,\quad D(\mathcal{H}_0)= D(\mathcal{H})
\end{equation}
and $n_{\pm}(\mathcal{H}_0)=1$. Next, we show that  $\mathcal{H}_0\geqq 0$. If we denote $L_j=-c_j^2\frac{d^2}{dx^2}+\cos(\phi_j)$ then from \eqref{trav21} we obtain 
\begin{equation}
\label{spec7}
L_j\psi=-\frac{1}{\phi'_j} \frac{d}{dx}\Big[c_j^2  (\phi'_j)^2 \frac{d}{dx}\Big(\frac{\psi}{\phi'_j}\Big)\Big],
\end{equation}
for any $\psi$. It is to be observed that  $\phi'_j\neq 0$. By using formula in \eqref{spec7} we have for any $\Lambda=(\psi_j)\in D(\mathcal{H}_0)$,
\begin{equation}\label{7spec}
\begin{split}
\langle \mathcal{H}_0 \Lambda,\Lambda\rangle= A+c_1^2\psi^2_1(0)\frac{\phi''_1(0)}{\phi'_1(0)}-\sum_{j=2}^3c_j^2\psi^2_j(0)\frac{\phi''_j(0)}{\phi'_j(0)}\equiv A+P,
\end{split}
\end{equation}
where $A\geqq 0$ represents the integral terms. Next we show that $P\geqq0$. Indeed, since $\phi_j''(0)\geqq 0$, for every $j$, $\phi_1'(0)>0$, and  $\phi_j'(0)<0$, for $j=2,3$, we obtain immediately $P\geqq 0$. Then, $\mathcal{H}_0\geqq 0$.

Due to  Proposition \ref{semibounded} (see Appendix \S A),  $n(\mathcal{L}_\lambda)\leqq 1$. Next, for $\Psi_{\lambda, \delta'}=(-\phi_1, \phi_2, \phi_3)\in D(\mathcal{L}_\lambda)$ (with a slight abuse of notation) we obtain
\begin{equation}\label{negaqua}
\langle \mathcal{L}_\lambda \Psi_{\lambda, \delta'}, \Psi_{\lambda, \delta'}\rangle=\int_{-\infty}^0[-\sin(\phi_1)+\cos(\phi_1)\phi_1]\phi_1dx +\sum_{j=2}^3\int_0^{+\infty}[-\sin(\phi_j)+\cos(\phi_j)\phi_j]\phi_jdx<0,
\end{equation}
because of $0<\phi_j(x)\leqq \pi$ and $x\cos x\leqq \sin x$ for all $x\in [0, \pi]$.  Then from minimax principle we  arrive at $n(\mathcal{L}_\lambda)=1$. This finishes the proof.
\end{proof}

\begin{remark}\label{fora}
For the case $\lambda\in (-\infty, -\frac{\pi}{2} \sum_{j=1}^3 c_j)$ in Proposition \ref{3main}, the formula for $P$ in \eqref{7spec}  satisfies $P<0$. Therefore, it is not clear whether the extension theory approach provides an estimate of the Morse-index of $\mathcal{L}_\lambda$; see also the related Remark 4.5 in \cite{AnPl-delta}.
\end{remark}

\begin{proof}[Proof of Theorem \ref{2main} (case $-\frac{\pi}{2} \sum_{j=1}^3 c_j \leqq \lambda \leqq -\sum_{j=1}^3 c_j$)]
From Propositions \ref{3main} and \ref{5main} we have $\ker( \mathcal{L}_\lambda)=\{0\}$ and $n( \mathcal{L}_\lambda)=1$. Thus, from Theorem \ref{cauchy2} and Theorem \ref{crit} there follows the instability property of the stationary profile $\Psi_{\lambda, \delta'}=(-\phi_1, \phi_2, \phi_3,0,0,0)$. Now, since the mapping data-solution for the sine-Gordon model on $H^1(\mathcal Y)\times L^2(\mathcal Y)$ is at least of class $C^2$ (indeed, it is smooth) by Theorem \ref{well2}, it follows that the linear instability property of $\Psi_{\lambda, \delta'}$   is in fact  of nonlinear type  in the $H^1(\mathcal Y)\times L^2(\mathcal Y)$-norm (see Henry {\it{et al.}} \cite{HPW82}, Angulo and Natali \cite{AngNat16}, and Angulo {\it{et al.}} \cite{ALN08}). This finishes the proof.
\end{proof}

\subsubsection{The spectral study in the case $\lambda\in (-\infty, -\frac{\pi}{2} \sum_{j=1}^3 c_j)$}
 
 In this subsection we study in more details the Morse index of $\mathcal{L}_\lambda$ for $\lambda\in (-\infty, -\frac{\pi}{2} \sum_{j=1}^3 c_j)$ via analytic perturbation theory. Our analysis specializes to the case $c_1=c_2=c_3$ in \eqref{ajs}; thus, we consider the following closed subspace of $L^2(\mathcal Y)$,
 \begin{equation}\label{C}
 \mathcal C_2=\{(u_j)_{j=1}^3\in L^2(\mathcal Y): \text{for}\; u(x)\equiv -u_1(-x), x>0, \;\text{we have}\; u(x)=u_2(x)=u_3(x),\; x>0\}.
\end{equation}
We immediately note from \eqref{ajs} that the soliton-profile belongs to $ \mathcal{C}_2$, $\Psi_{\lambda, \delta'}=(-\phi_1, \phi_2, \phi_3)\in \mathcal C_2$, with $\phi_i=\phi_{i, a_i(\lambda)}$ and  $-a_1(\lambda)=a_2(\lambda)=a_3(\lambda)$, where $a_2(\lambda)$ is determined by \eqref{deriva4}. Our strategy here will be to apply the linear instability criterion in Theorem \ref{crit} within the space $(H^1(\mathcal Y)\cap \mathcal C_2)\times \mathcal C_2$. Thus we start with the verification of Assumption $(S_1)$.

\begin{proposition}\label{semi2} Let us consider the $C_0$-semigroup $\{G(t)\}_{t\geqq 0}$ on $H^1(\mathcal Y)\times L^2(\mathcal Y )$ defined by \eqref{FRW}. Then,
\begin{enumerate}
\item[1)] For all $t\geqq 0$, the subspace $(H^1(\mathcal Y)\cap \mathcal C_2)\times \mathcal C_2$ is invariant under $G(t)$. Moreover, the infinitesimal generator of $G(t)$ is the operator $\mathcal A=JE$ with $D(\mathcal A)=(D(\mathcal H_\lambda)\cap \mathcal C_2)\times \mathcal C_2$.

\item[2)] Assumption $(S_1)$ is satisfied; more precisely, the operator $J\mathcal E$ is the generator of a $C_0$-semigroup $\{S(t)\}_{t\geqq 0}$ on $(H^1(\mathcal Y)\cap \mathcal C_2)\times \mathcal C_2$  with $D(J\mathcal E)=(D(\mathcal H_\lambda)\cap \mathcal C_2)\times \mathcal C_2$.
\end{enumerate}
 \end{proposition}  
 \begin{proof}
\noindent 1) By the representation formula for $G(t)$ in \eqref{FRW}, it is sufficient to show that the resolvent operator for $\mathcal A$, $R(\eta:\mathcal A)$, satisfies $R(\eta:\mathcal A)((H^1(\mathcal Y)\cap \mathcal C_2)\times \mathcal C_2)\subset (D(\mathcal H_\lambda)\cap \mathcal C_2)\times \mathcal C_2$. Indeed, initially for $\lambda\in \mathbb R-\{0\}$ we can see, similarly as in the proof of Theorem 2.2 in \cite{AnPl-delta}, that for $\eta \in \mathbb C$ such that $-\eta^2 \in \rho(\mathcal H_\lambda)$ we obtain for $\Psi=(\bold u, \bold v)\in H^1(\mathcal Y)\times L^2(\mathcal Y)$ the representation
\begin{equation}\label{resolA}
R(\eta:\mathcal A)\Psi=\left(\begin{array}{c} -R(-\eta^2: \mathcal H_\lambda)(\eta \bold u+\bold v) \\
-\eta R(-\eta^2: \mathcal H_\lambda)(\eta \bold u+\bold v) -\bold u \end{array}\right),
\end{equation}
where  $R(-\eta^2: \mathcal H_\lambda)=(-\eta^2I_3- \mathcal H_\lambda)^{-1}: L^2(\mathcal Y)\to D(\mathcal H_\lambda)$. Thus, we only need to show that $R(-\eta^2: \mathcal H_\lambda)$ satisfies $R(-\eta^2: \mathcal H_\lambda)(\mathcal C_2)\subset D(\mathcal H_\lambda)\cap \mathcal C_2$. It is not difficult to see that $\mathcal H_\lambda(D(\mathcal H_\lambda)\cap \mathcal C_2)\subset \mathcal C_2$. Now, a explicit representation for $R(-\eta^2: \mathcal H_\lambda)$ for any $\eta>0$  (without loss of generality)  and $\lambda>0$ can be obtained via the following formulas: for $\bold u=(u_j)_{j=1}^3\in L^2(\mathcal Y)$ and $(\Phi_j)_{j=1}^3=(\mathcal H_\lambda +\eta^2I_3)^{-1} \bold u$ ($c_j>0$ without loss of generality)
\begin{enumerate}
\item[(a)] for $x<0$
\begin{equation}
\label{formu1}
\Phi_1(x)=(-c_1\frac{d^2}{dx^2}+\eta^2)^{-1}(u_1)(x)=\frac{d_1}{c_1}e^{\frac{\eta}{ \sqrt{c_1} }x} +\frac{1}{2 \sqrt{c_1} \eta}\int_{-\infty}^0 u_1(y) e^{-\frac{\eta}{ \sqrt{c_1} } |x-y|} dy
\end{equation}
\item[(b)] for $x>0$ and $j=2,3$,
\begin{equation}
\label{formu2}
\Phi_j(x)=(-c_j\frac{d^2}{dx^2}+\eta^2)^{-1}(u_j)(x)=\frac{d_j}{c_j}e^{-\frac{\eta}{ \sqrt{c_j} }x} +\frac{1}{2 \sqrt{c_j}\eta}\int_0^{\infty} u_j(y) e^{-\frac{\eta}{ \sqrt{c_j} } |x-y|} dy,
\end{equation}
where the constants $d_j=d_j(\eta, (\Phi_j))$ are chosen  such that $(\Phi_j)\in D(\mathcal H_{\lambda})$. Moreover, for $\lambda>0$, it is not difficult to see that for $(u_j)\in \mathcal C_2$ we obtain $(\Phi_j)\in \mathcal C_2$.
Next, for the case $\lambda<0$ we need to use Theorem \ref{spectrum2}. We note that the eigenfunction $\Phi_\lambda=(-e^{-\frac{\alpha}{c_1}x}, e^{\frac{\alpha}{c_1}x}, e^{\frac{\alpha}{c_1}x})$ for $\mathcal H_\lambda$ associated with the eigenvalue $\theta_0=-9c_1^2/\lambda^2$ and with $\alpha=3c_1/\lambda<0$, obviously belongs to $\mathcal C_2$ (we recall that $c_1=c_2=c_3>0$). Thus, by using similar formulae to those in \cite{AnPl-delta} (specifically, formulae (2.9) and (2.10) in that reference) we immediately obtain that $R(-\eta^2: \mathcal  H_\lambda)\bold u=(\Psi_j)_{j=1}^3\in \mathcal C_2$ with $\eta^2\neq -\theta_0$.
\end{enumerate} 
\noindent 2) Consider $J\mathcal E=J E+\mathcal M$ with
 \[
  \mathcal M=\left(\begin{array}{cc} 0& 0 \\(-\cos(\phi_j)\delta_{j,k})& 0\end{array}\right).
  \]
 We know that $ \mathcal M$ is a bounded operator on $H^1(\mathcal Y)\times L^2(\mathcal Y)$. Now, for $\bold u\in \mathcal C_2$ we can see that $(-\cos(\phi_j)\delta_{j,k})\bold u\in \mathcal C_2$ and so $\mathcal M: (H^1(\mathcal Y)\cap \mathcal C_2)\times \mathcal C_2\to (H^1(\mathcal Y)\cap \mathcal C_2)\times \mathcal C_2$ is well defined. Therefore, since $JE$ generates a  $C_0$-semigroup on $(H^1(\mathcal Y)\cap \mathcal C_2)\times  \mathcal C_2$, it follows from standard semigroup theory that $J\mathcal E$ also has this property (see Pazy \cite{Pa}).
\end{proof}

 \begin{proposition}\label{6main}
Let $\lambda_0=-\frac{\pi}{2} \sum_{j=1}^3 c_j$ and consider  $\mathcal B=\mathcal C_2 \cap D (\mathcal{L}_{\lambda_0})$.  Then $ \mathcal{L}_{\lambda_0}  : \mathcal B\to \mathcal C_2$ is well defined and  $\ker( \mathcal{L}_{\lambda_0} |_\mathcal B )=\{0\}$ and $n( \mathcal{L}_{\lambda_0} |_\mathcal B )=1$.  
\end{proposition}

\begin{proof} Initially for  $\phi_i=\phi_{i, a_i(\lambda_0)}$, $ a_i(\lambda_0)=0$. Then we have (with a slight  abuse  of notation) $\Psi_{\lambda_0, \delta'}=(-\phi_1, \phi_2, \phi_3)\in \mathcal B$. Next, suppose $(u_1, u_2,u_3)\in \mathcal B\cap \ker( \mathcal{L}_{\lambda_0})$. Then from Proposition \ref{4main} there exist $\theta, \mu\in \mathbb R$ such that  
$$
(u_1, u_2,u_3)=\theta \Phi_1+\mu \Phi_2.
$$
Hence, since $-u_1(0-)=u_2(0+)=u_3(0+)$ we obtain $\theta=-2\mu$ and  $u_1(0-)=-\theta \phi_1'(0-)=-u_3(0+)=\mu \phi_3'(0+)=-\mu\phi_1'(0-)$. Therefore $\theta=\mu$ and so $\mu=\theta=0$. This shows that $\ker( \mathcal{L}_{\lambda_0} |_\mathcal B )$ is trivial. Lastly, from \eqref{negaqua} we have $\langle \mathcal{L}_{\lambda_0} \Psi_{\lambda_0, \delta'}, \Psi_{\lambda_0, \delta'}\rangle <0$. Therefore, in view of Proposition \ref{5main}, we finish the proof.
\end{proof}

\begin{remark}\label{extension}
The spectral structure of the operator $ \mathcal{L}_{\lambda_0} |_\mathcal B $ given in   Proposition \ref{6main} clearly depends on the choice of the subspace $\mathcal C_2$ in \eqref{C}. For instance, if we consider the case $c_2=c_3$ (or still $c_1=c_2=c_3$) and the subspace $ \mathcal C_1=\{(u_j)_{j=1}^3\in L^2(\mathcal Y):u_2=u_3\}$, it is not difficult to see that for  $\mathcal B_1=\mathcal C_1 \cap D (\mathcal{L}_{\lambda_0})$ we have $\dim(\ker( \mathcal{L}_{\lambda_0} |_{\mathcal B_1} ))=1$ and $n( \mathcal{L}_{\lambda_0} |_{\mathcal B_1} )=1$. In this case, we cannot apply  Theorem \ref{crit}. 
\end{remark}

The following result is a natural consequence from Propositions \ref{3main}, \ref{5main} and \ref{6main}.

\begin{proposition}\label{7main}
Let $\lambda\in (-\infty, -\frac{\pi}{2} \sum_{j=1}^3 c_j)$ and consider  $\mathcal B=\mathcal C_2 \cap D (\mathcal{L}_{\lambda})$.  Then $ \mathcal{L}_{\lambda}  : \mathcal B\to \mathcal C_2$ is well defined and $n( \mathcal{L}_{\lambda} |_\mathcal B )=1$.
\end{proposition}
\begin{proof} The proof is based on Proposition \ref{6main}, analytic perturbation theory around $\lambda_0$, a principle of continuation based in the Riesz-projection, as well as on the ideas in the proof of Proposition 4.4 in \cite{AnPl-delta} (see also Proposition \ref{antimain4} below). Details are left to the reader.
\end{proof}

\begin{remark}\label{extension2}
 We note that via analytic perturbation theory is possible to see  that for the case $c_2=c_3$, the subspace $\mathcal C_1$ defined in Remark \ref{extension}, 
$\mathcal B_3=\mathcal C_1 \cap D (\mathcal{L}_{\lambda})$ and  $\lambda\in (-\infty, -\frac{\pi}{2} \sum_{j=1}^3 c_j)$, we have
$n( \mathcal{L}_{\lambda} |_{\mathcal B_3} )=2$. In this case we do not know what happens with the stability properties of $\Psi_{\lambda, \delta'}$, but we conjecture that they are unstable (see \cite{AngCav}). It is to be noted that Proposition \ref{semi2} is still true on the subspace $(H^1(\mathcal Y)\cap \mathcal C_1)\times  \mathcal C_1$.
\end{remark}

\begin{proof}[Proof of Theorem \ref{2main} (case $ - \infty < \lambda \leqq -\frac{\pi}{2} \sum_{j=1}^3 c_j, \, c_1=c_2=c_3$)]
From Propositions \ref{3main}, \ref{6main} and \ref{7main} we have $\ker( \mathcal{L}_\lambda|_\mathcal B)=\{0\}$ and $n( \mathcal{L}_\lambda|_\mathcal B)=1$. Moreover, Proposition \ref{semi2} verifies Assumption $(S_1)$ in the linear instability criterion in subsection 3.1. Thus, from Theorem \ref{crit} follows the linear instability property of the stationary profile $\Phi_{\lambda, \delta'}$. Lastly, from Theorem \ref{well2} and the analysis of Henry {\it et al.} in \cite{HPW82} (see also Angulo {\it{et al.}} \cite{ALN08}, Angulo and Natali \cite{AngNat16}) we obtain the nonlinear instability property of $\Phi_{\lambda, \delta'}$. This finishes the proof.
\end{proof}

\subsection{Kink/anti-kink instability theory on a $\mathcal{Y}$-junction of type I}
\label{secAK}

In this subsection we study the existence and stability of the kink/anti-kink profiles defined in \eqref{antikink}.  Since these stationary profiles do not belong to the classical $H^2(\mathcal Y)$ Sobolev space, we need to make precise the functional spaces suitable for our needs. 

\subsubsection{The kink/anti-kink solutions with specific profile on a $\mathcal Y$-junction}
\label{secAK-K}

Let us consider the specific class of kink/anti-kink solutions defined in \eqref{antikink} satisfying the $\delta'$-condition in \eqref{bcIfam} with the continuity property $\phi_2(0+)= \phi_3(0+)$ and, for simplicity, subject to the condition
\begin{equation}
\label{equalcs}
c_1 = c_2 = c_3 = 1.
\end{equation}
Consequently we have $a_2=a_3$ and, therefore, $\phi_2=\phi_3$ on $(0,+\infty)$.  Next, the continuity condition $\phi_1'(0-)=\phi'_2(0+)$ implies that $\cosh(a_1)=\cosh(a_2)$. As we are interested in non-continuous profiles at the vertex $\nu=0$ of the $\mathcal Y$-junction, we consider the case $a_1=-a_2$. Now the Kirchhoff's type condition in \eqref{bcIfam} implies the following equality for $y=e^{-a_1}$
\begin{equation}\label{anticondi}
F(y)\equiv -\frac{1+y^2}{y}[2 \arctan(y) - \arctan (1/y)]=\lambda.
\end{equation}
Thus, we obtain immediately the following behavior of the mapping $F$: $(i)$ $F'(y)<0$ for all $y>0$; $(ii)$ $F(y^*)=0$ for a unique $y^*\in (0,1)$, and $(iii)$, $\lim_{y\to 0^+} F(y)=+\infty$, $\lim_{y\to +\infty} F(y)=-\infty$. Then, from \eqref{anticondi} we have the following specific behavior of the $\lambda$-parameter:
\begin{enumerate}
\item[a)] for $a_1=0$, $\lambda=-\frac{\pi}{2}$,
\item[b)] for $a_1>0$, $\lambda\in (-\frac{\pi}{2}, +\infty)$,
\item[c)] for $a_1<0$, $\lambda\in (-\infty, -\frac{\pi}{2})$.
\end{enumerate}

Henceforth, from \eqref{anticondi} and the properties for $F$ we obtain the existence of a smooth shift-map (also real analytic), $
\lambda\in (-\infty, +\infty)\mapsto a_1(\lambda)$ satisfying \eqref{anticondi}, and such that the mapping 
$$
\lambda\in \Big(-\infty, +\infty\Big)\mapsto \Pi_{\lambda, \delta'}=(\phi_{1,a_1(\lambda)}, \phi_{2, a_2(\lambda)}, \phi_{2, a_2(\lambda)},0,0,0),\quad a_2(\lambda)=-a_1(\lambda)
$$
 represents a real-analytic family of static profiles for the sine-Gordon equation on a $\mathcal Y$-junction of first-type satisfying  for all $\lambda \in \mathbb R$, $\lim_{x\to -\infty} \phi_{1,a_1(\lambda)}(x)=2\pi$. Hence we obtain, for $a_i=a_i(\lambda)$ and $\phi_i=\phi_{i,a_i(\lambda)}$, the following behavior:
\begin{enumerate}
\item[1)] for  $\lambda\in (-\infty, -\frac{\pi}{2} )$ we have $a_1<0$:  therefore   $ \phi'_i<0$ and $\phi_1''(a_1) =0=\phi_2''(-a_1) $. Moreover, $\phi_1(0)\in (0, \pi)$ and  $\phi_2(0)\in (\pi, 2\pi)$.
Thus, the profile of $(\phi_1, \phi_2,\phi_2)$, looks similar to the one shown in Figure \ref{figAKBumpYI} below (bump-profile type); 
\item[2)] for   $\lambda\in (-\frac{\pi}{2}, +\infty)$ we obtain $a_1>0$:  therefore   $ \phi''_1<0$ and $ \phi''_i>0$ for $i=2,3$,   Thus, the profile of $(\phi_1, \phi_2,\phi_3)$ looks similar to that in Figure \ref{figAKTailYI} below (typical tail-profile); 
\item[3)]  the case  $\lambda=-\frac{\pi}{2}$ implies $a_1=a_2=a_3=0$: Therefore,  $\phi_1(0)=\phi_2(0) =\phi_3(0)=\pi$ (continuity at the vertex of the graph),
$\phi_i''(0) =0$, $i=1,2, 3$. Thus, the profile of $(\phi_1, \phi_2,\phi_3)$ represents exactly two kink/anti-kink solitons profiles connected in the vertex of the graph such as shown in Figure \ref{figAKZeroWYI}.
\end{enumerate}

\begin{figure}[t]
\begin{center}
\subfigure[$\lambda\in (-\infty, -\frac{\pi}{2})$]{\label{figAKBumpYI}\includegraphics[scale=.4, clip=true]{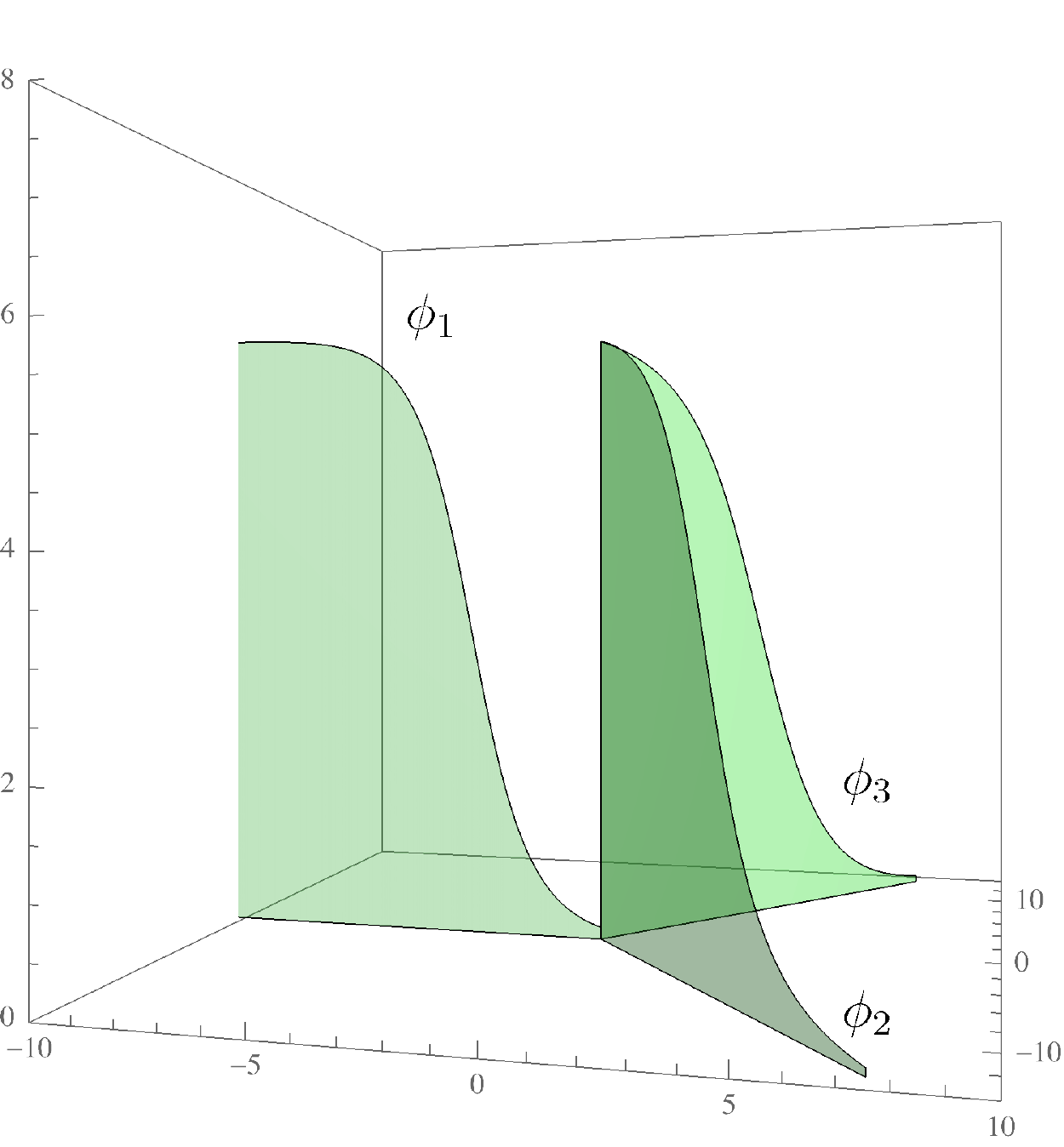}}
\subfigure[$\lambda\in (-\frac{\pi}{2}, \infty)$]{\label{figAKTailYI}\includegraphics[scale=.4, clip=true]{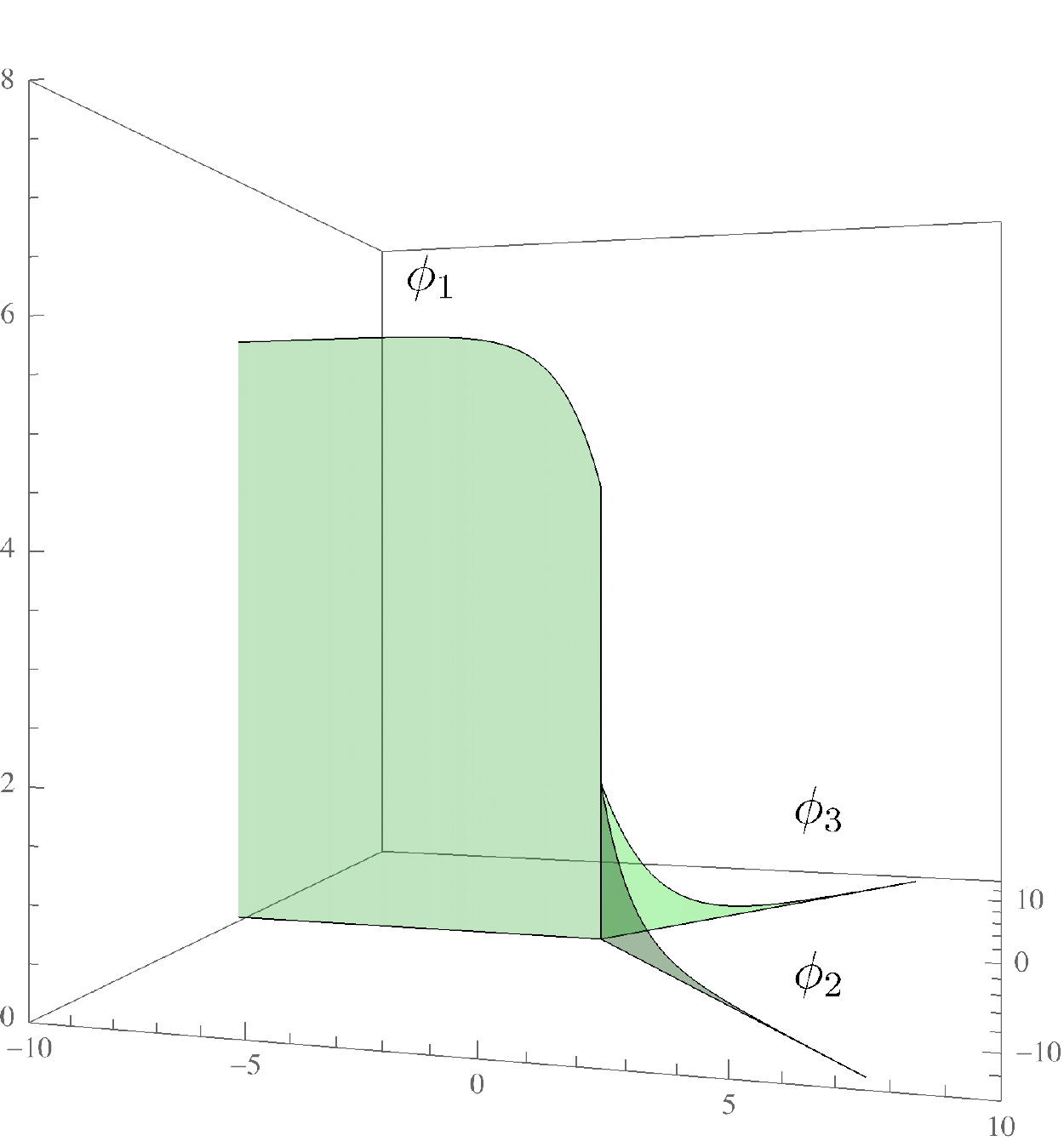}}
\subfigure[$\lambda=-\frac{\pi}{2}$]{\label{figAKZeroWYI}\includegraphics[scale=.4, clip=true]{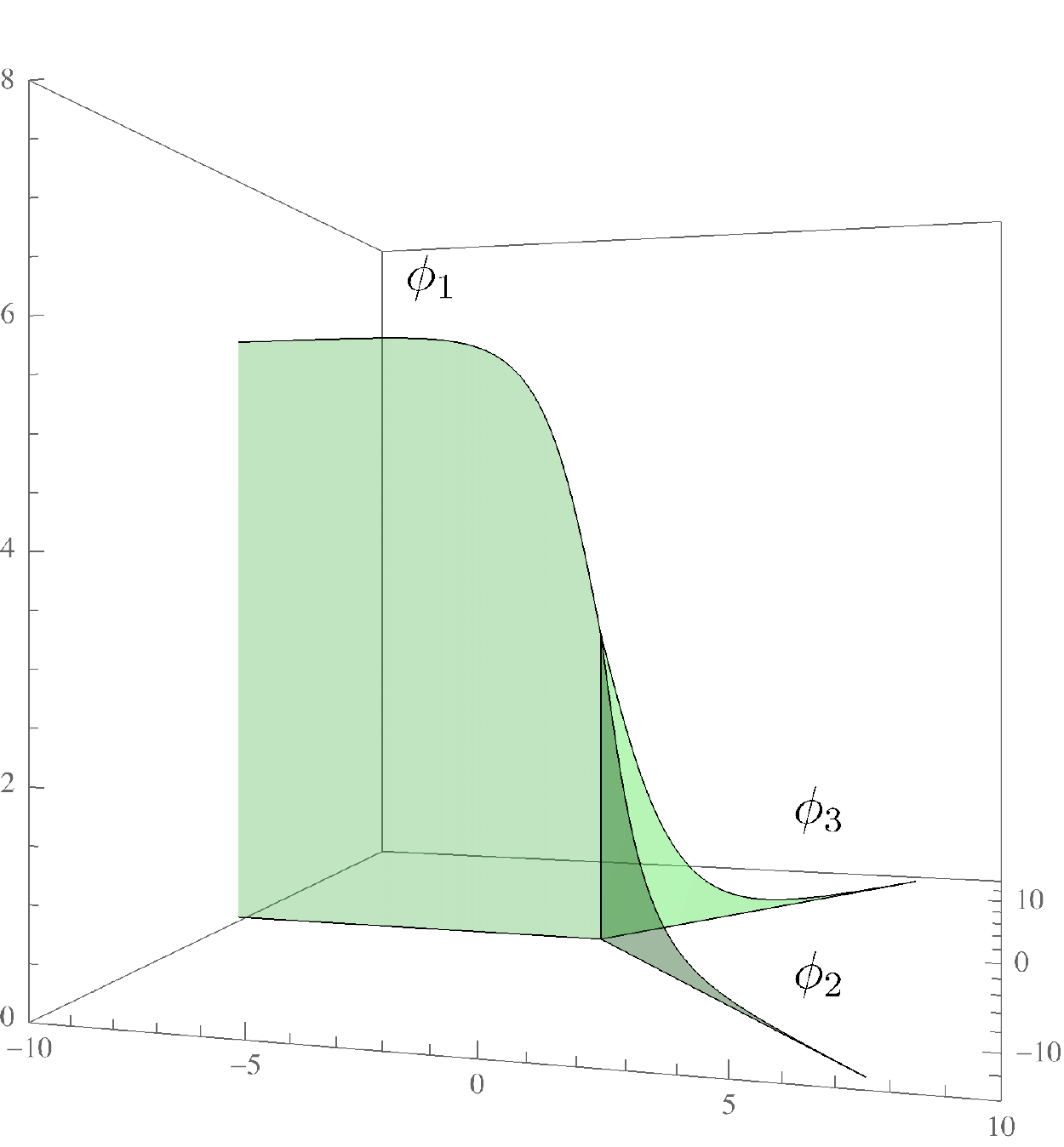}}
\end{center}
\caption{\small{Plots of stationary solutions $(\phi_1, \phi_2, \phi_3)$ defined in \eqref{antikink} in the case where $c_j = 1$ for all $j=1,2,3$, for different values of $\lambda \in \R$. Panel (a) shows the stationary profile solutions (``bump-type'' configuration) for the case $\lambda = 100 \in (-\infty,-\pi/2)$. The shaded profile of the anti-kink on the edge $E_1 = (-\infty,0)$ illustrates the fact that the profile $\phi_1(x)$ has infinite mass as $\phi_1 \notin L^2(-\infty,0)$. Panel (b) shows the profile of\/ ``tail-type'' for the case $\lambda = 1 \in (-\pi/2,\infty)$. Notice the discontinuity at the vertex. Panel (c) shows the profile solutions when $\lambda = -\pi/2$ (color online).}}\label{figAKsolitons}
\end{figure}

\begin{remark}
\label{cj} 
It is to be observed that we have left open the description of other kink/anti-kink-soliton profiles not satisfying the continuity property at zero for the components $\phi_2, \phi_3$, as well as the case where the constants $c_i$ are not all equal to each other. These other profiles, however, can be studied following the methods to be described in this section.
\end{remark}

Our  instability result for the kink/anti-kink profiles
$\Pi_{\lambda,\delta'}=(\phi_1, \phi_2, \phi_2, 0,0,0)$ (with a slight abuse of notation) with $\phi_i=\phi_{i, a_i(\lambda)}$ defined in \eqref{antikink} is the following
\begin{theorem}\label{0antimain} 
Let $\lambda\in  (-\frac{\pi}{2}, +\infty)$ and the smooth family of stationary kink/anti-kink profiles $\lambda\to \Pi_{\lambda,\delta'}$ determined above. Then $\Pi_{\lambda, \delta'}$ is spectrally and nonlinearly unstable for the sine-Gordon model \eqref{sg2} on a $\mathcal Y$-junction  of first type. 
\end{theorem}

The proof of Theorem \ref{0antimain} will follow by combining   the framework of extension theory and the analytic perturbation theory.  At this point, some observations about the stability problem for values of $\lambda$ outside  the range $(-\frac{\pi}{2}, +\infty)$ are in order. Indeed, for $\lambda=-\frac{\pi}{2}$ we show that the fundamental Schr\"odinger diagonal operator $\mathcal L$ in \eqref{trav23} associated to the kink/anti-kink profiles $\Pi_{\lambda,\delta'}$ has a two-dimensional kernel and a  Morse index  equal to one.  For $\lambda<-\frac{\pi}{2}$ we do not know which is exactly the Morse index for $\mathcal L$, but for completeness and for  future study, we established in  Proposition \ref{Appanti} in Appendix \S B that, in this case, the Morse index is at least two. Thus, for  values of $\lambda$ outside  $(-\frac{\pi}{2}, +\infty)$ the stability properties of the kink/anti-kink profiles \eqref{antikink} remain open.

\subsubsection{Functional space for stability properties of the kink/anti-kink profile}
\label{akfunspa}
The natural framework space for studying  stability properties associated to the kink/anti-kink soliton profile $\Phi=(\phi_j)_{j=1}^3$ described in the former subsection for the sine-Gordon model is $\mathcal X(\mathcal Y)= H^1_{\mathrm{loc}}(-\infty, 0)\bigoplus H^1(0, \infty)\bigoplus H^1(0, \infty)$. Thus we say that a flow $t\to (u(t), v(t))\in \mathcal X(\mathcal Y) \times L^2(\mathcal Y)$ is called  a \emph{perturbed solution} for the  kink/anti-kink profile $\Phi\in \mathcal X(\mathcal Y)$ if for $(P(t), Q(t))\equiv (u(t)-\Phi,v(t))$ we have that $ (P(t), Q(t))\in H^1(\mathcal Y)\times L^2(\mathcal Y)$ and  
$\bold z=(P,Q)^\top$ satisfies the following vectorial perturbed sine-Gordon model 
 \begin{equation}\label{akmodel}
 \begin{cases}
  \bold z_t=JE\bold z +F_1(\bold z)\\
 P(0)=u(0)-\Phi \in H^1(\mathcal Y),\\
 Q(0)=v(0)\in  L^2(\mathcal Y),
  \end{cases}
 \end{equation} 
  where for $P=(p_1, p_2, p_3)$ we have
   \begin{equation}
  F_1(\bold z)=\left(\begin{array}{cc}  0\\  0   \\  0   \\  \sin(\phi_1)-\sin (p_1+\phi_1)   \\  \sin(\phi_2)-\sin (p_2+\phi_2)   \\  \sin(\phi_3)-\sin (p_3+\phi_3)  \end{array}\right).
  \end{equation} 
Then, by studying stability  properties of the stationary anti-kink $\Pi_{\lambda, \delta'}=(\phi_1,\phi_2,\phi_2,0,0,0)$ by the sine-Gordon model on $\mathcal X(\mathcal Y) \times L^2(\mathcal Y)$ is reduced to study stability properties of the trivial solution $(P,Q)=(0,0)$ for the  linearized  model \eqref{akmodel} around $(P,Q)=(0,0)$. Thus, via Taylor's Theorem  we obtain the  linearized system in \eqref{stat4} but with the Schr\"odinger diagonal operator $\mathcal L$ in \eqref{stat5} now determined by the anti-kink profile $\Phi=(\phi_j)$. We will denote this operator by $\mathcal L_{\lambda}$ in \eqref{deriva1} below, with the domain $D(\mathcal L_{\lambda})$ in \eqref{I3trav23}. In this form, we can apply {\it{ipsi litteris}} the  semi-group theory results in subsection \ref{seclocalWP}  to the operator $JE$ and to the local well-posedness problem in $H^1(\mathcal Y)\times L^2(\mathcal Y)$ for the vectorial perturbed sine-Gordon model \eqref{akmodel}. Lastly, we note that the kink/anti-kink profile $\Phi\in \mathcal X(\mathcal Y)$ but $\Phi'\in H^2(\mathcal Y)$.

\subsubsection{The spectral study in the kink/anti-kink case}

In this subsection we give the spectral informations for the family of self-adjoint operators $(\mathcal{L}_\lambda, D(\mathcal{L}_\lambda))$ where
\begin{equation}\label{deriva1}
\mathcal{L}_\lambda =\Big (\Big(-\frac{d^2}{dx^2}+\cos(\phi_j)
\Big)\delta_{j,k} \Big ),\quad \l1\leqq j, k\leqq 3,
\end{equation}
associated to the kink/anti-kink solutions $(\phi_1, \phi_2, \phi_2)$ determined in the previous subsection. Here $D(\mathcal{L}_\lambda)$ is the $\delta'$-interaction domain defined  in \eqref{I3trav23} (see also Proposition \ref{2M} at Appendix \S A).

We begin by proving a result that applies to all values of $\lambda$ under consideration.
\begin{proposition}\label{antimain}
Let $\lambda\in  (-\infty,+\infty)$. Then, $\ker( \mathcal{L}_\lambda )=\{\mathbf{0}\}$ for all $\lambda\neq -\frac{\pi}{2}$. For $\lambda=-\frac{\pi}{2}$ we have $\dim (\ker( \mathcal{L}_\lambda) )=2$. Moreover, for all $\lambda$ we obtain $\sigma_{\mathrm{ess}}(\mathcal{L}_\lambda)=[1,+\infty)$.
\end{proposition}

\begin{proof} Let $\bold u=(u_1, u_2, u_3)\in D(\mathcal{L}_\lambda)$ and $\mathcal{L}_\lambda \bold{u}=\bold{0}$. Then, since $(\phi'_1, \phi'_2, \phi'_2)\in H^2(\mathcal Y)$ follows from Sturm-Liouville theory on half-lines 
\begin{equation}\label{anti2spec}
u_1(x)=\alpha_1\phi'_1(x), \;\;x<0,\quad u_j(x) = \alpha_j\phi'_j(x), \;\;x>0,\;\; j=2,3,
\end{equation}
for some $\alpha_1$ and $\alpha_j$, $j = 2,3$, real constants. Next, we consider the following cases:
\begin{enumerate}
\item[a)] Suppose $a_1>0$ ($\lambda\in [0, +\infty)$): The conditions $u'_1(0-)=u'_2(0+)=u'_3(0+)$ and $\phi_2=\phi_3$ imply that $\alpha_2=\alpha_3$, since $\phi_2''(0)\neq 0$. Next, the jump-condition and $\phi'_1(0-)=\phi'_2(0+)=\phi'_3(0+)$ imply
\begin{equation}\label{anti3}
(2\alpha_2-\alpha_1)\phi'_1(0)=\lambda \alpha_1 \phi''_1(0).
\end{equation}
Now, since for all $\lambda\neq -\frac{\pi}{2}$ we have $\phi''_1(0)=-\phi''_2(0)\neq 0$. Thus, $\alpha_1=-\alpha_2$ and from \eqref{anti3} we arrive at the relation $-3\alpha_1\phi'_1(0)=\lambda \alpha_1 \phi''_1(0)$. Suppose $\alpha_1\neq 0$, then $\lambda=-3\phi'_1(0)/\phi''_1(0)$. Since $\phi'_1(0)<0$ and $\phi''_1(0)<0$ we get $\lambda<0$, which is a contradiction. So, we need to have $0=\alpha_1=\alpha_2=\alpha_3$ and therefore $ \bold{u}=\bold{0}$.

\item[b)] Suppose $a_1>0$ ($\lambda\in (-\frac{\pi}{2}, 0)$): From item $a)$ we still have the relation $-3\alpha_1\phi'_1(0)=\lambda \alpha_1 \phi''_1(0)$. Suppose $\alpha_1\neq 0$.
Then, by the formula for the anti-kink on $(-\infty, 0)$ in \eqref{antikink} we obtain for $a_1>0$,
$$
\frac{\phi'_1(0)}{\phi''_1(0)}=\frac{\cosh(a_1)}{\sinh(a_1)}>1. 
$$
Thus, we obtain $\lambda<-3$ which is a contradiction. Therefore, $\alpha_1=0$ and so $ \bold{u}=\bold{0}$.

\item[c)] Suppose $a_1<0$ ($\lambda\in (-\infty, -\frac{\pi}{2})$): From item $a)$ we still have the relation $-3\alpha_1\phi'_1(0)=\lambda \alpha_1 \phi''_1(0)$. Suppose $\alpha_1\neq 0$. Since $\phi'_1(0)<0$ and $\phi''_1(0)>0$ we get $\lambda>0$ which is a contradiction. So, we need to have $0=\alpha_1=\alpha_2=\alpha_3$ and therefore $ \bold{u}=\bold{0}$.

\item[d)] Suppose $a_1=0$ ($\lambda=-\frac{\pi}{2}$): In this case the Kirchhoff's condition for $\bold u$, $\phi'_1(0)=\phi'_2(0)$ and $\phi''_1(0)=0$ we get the relation $\alpha_2+\alpha_3-\alpha_1=0$. Therefore
$$
(u_1, u_2, u_3)=\alpha_2(\phi'_1, \phi'_2,0)+ \alpha_3(\phi'_1, 0,  \phi'_2).
$$
Since $(\phi'_1, \phi'_2,0), (\phi'_1, 0,  \phi'_2)\in D(\mathcal L_\lambda)$ we obtain that $\dim(\ker(\mathcal L_\lambda))=2$.

\end{enumerate}

The statement $\sigma_{\mathrm{ess}}(\mathcal{L}_\lambda)=[1,+\infty)$ is an immediate consequence of Weyl's Theorem (cf. \cite{RS4}) because of $\lim_{x\to -\infty} \cos(\phi_1(x))=1=\lim_{x\to +\infty} \cos(\phi_2(x))$. This finishes the proof.
\end{proof}

\begin{proposition}\label{anti5main}
Let $\lambda\in [-\frac{\pi}{2}, 0)$. Then $n( \mathcal{L}_\lambda )=1$.
\end{proposition}

\begin{proof}  From Proposition \ref{2M} in Appendix \S A we have that the family $(\mathcal{L}_\lambda, D(\mathcal{L}_\lambda))$ represents all the self-adjoint extensions of the closed symmetric operator, $(\mathcal H_1, D(\mathcal H_1))$, where
\begin{equation}\label{6spec}
\mathcal{H}_1=\Big (\Big(-\frac{d^2}{dx^2}+\cos(\phi_j)
\Big)\delta_{j,k} \Big ),\;\l1\leqq j, k\leqq 3,\quad D(\mathcal{H}_1)= D(\mathcal{H})
\end{equation}
and $n_{\pm}(\mathcal{H}_1)=1$. Next, we show that  $\mathcal{H}_1\geqq 0$. Indeed, by using a similar argument as in the proof of Proposition \ref{5main} we have for any $\Lambda_1=(\psi_j)\in D(\mathcal{H}_0)$,
\begin{equation}\label{anti7spec}
\begin{split}
\langle \mathcal{H}_1 \Lambda_1,\Lambda_1\rangle= A_1+\psi^2_1(0)\frac{\phi''_1(0)}{\phi'_1(0)}-\sum_{j=2}^3\psi^2_j(0)\frac{\phi''_j(0)}{\phi'_j(0)}\equiv A_1+P_1,
\end{split}
\end{equation}
where $A_1$ represents the non-negatives integral terms. Next for $\lambda\in [-\frac{\pi}{2}, +\infty)$  ($a_1\geqq 0$) we obtain immediately that $P_1\geqq 0$, so $\langle \mathcal{H}_1 \Lambda_1,\Lambda_1\rangle\geqq 0$. Due to  Proposition \ref{semibounded} (see Appendix \S A) we obtain  $n(\mathcal{L}_\lambda)\leqq 1$. 

In the sequel we show that $n(\mathcal{L}_\lambda)\geqq 1$ for $\lambda\in [-\frac{\pi}{2}, 0)$. Indeed, we consider the quadratic form $\mathcal Q$ associated to $(\mathcal{L}_\lambda, D(\mathcal{L}_\lambda))$ for  $\Lambda =(\psi_i)\in H^1(\mathcal Y)$ by
\begin{equation}\label{antiqua}
\mathcal Q(\Lambda)=\frac{1}{\lambda} \Big(\sum_{j=2}^3 \psi_j(0)-\psi_1(0)\Big)^2 +\int_{-\infty}^0(\psi'_1)^2 + \cos(\phi_1) \psi_1^2 \, dx +\sum_{j=2}^3 \int_0^{\infty}(\psi'_j)^2 + \cos(\phi_j) \psi_j^2 \, dx.
\end{equation}
Next, for $\Lambda_1=(\phi'_1, \phi'_2, \phi'_2)\in H^1(\mathcal Y)$ we obtain  from the equalities  $-\phi'''_j + \cos(\phi_j)\phi'_j=0$, $\phi'_1(0)=\phi'_2(0)=\phi'_3(0)$, and integration by parts the relation
\begin{equation}\label{2antiqua}
\mathcal Q(\Lambda_1)=\frac{1}{\lambda} [\phi'_1(0)]^2 +\phi'_1(0)[\phi''_1(0)-2\phi''_2(0)].
\end{equation}
Thus, for $\lambda=-\frac{\pi}{2}$ we have $\phi''_j(0)=0$ and therefore $\mathcal Q(\Lambda_1)<0$. 

Now, for $\lambda\in (-\frac{\pi}{2}, 0)$ we consider $\Lambda_2=(0, \phi'_2, \phi'_2)\in H^1(\mathcal Y)$. Then
\begin{equation}\label{3antiqua}
\mathcal Q(\Lambda_2)=\frac{1}{\lambda} [2\phi'_2(0)]^2 -2\phi'_2(0)\phi''_2(0)<0
\end{equation}
because of $\mathcal Q(\Lambda_2)<0$ if and only if $2\frac{\phi'_2(0)}{\phi''_2(0)}<\lambda$  if and only if $-2\frac{\phi'_1(0)}{\phi''_1(0)}<\lambda$. But, since $-2\frac{\phi'_1(0)}{\phi''_1(0)}<-2<\lambda$ we get \eqref{3antiqua}. This finishes the proof.
\end{proof}

\begin{remark}\label{antifora}
For the case $\lambda\in [0, +\infty)$ in Proposition \ref{anti5main}, it was not possible for us to show in an easy way that the quadratic form $\mathcal Q$ in \eqref{antiqua} has a negative direction. But we will see below, via analytic perturbation approach, that we still have $n( \mathcal{L}_\lambda )=1$. We note that the formula for $P_1$ in \eqref{anti7spec}  satisfies $P_1<0$ for all $\lambda\in (-\infty, -\frac{\pi}{2})$
\end{remark}

\begin{proposition}
\label{antimain4}
Let $\lambda\in [0, +\infty)$. Then $n( \mathcal{L}_\lambda )=1$.
\end{proposition}

\begin{proof} We use analytic perturbation theory. Initially, from subsection \ref{secAK-K} we have the real-analytic  mapping function $\lambda\in (-\infty ,+\infty)\to a_1(\lambda)$ satisfying
\begin{equation}
a_1(\lambda)=\begin{cases}
\begin{aligned}
&<0,\quad \text{for}\;\;\lambda\in (-\infty,-\frac{\pi}{2}),\\
&=0,\quad \text{for}\;\;\lambda=-\frac{\pi}{2}\\
&>0,\quad \text{for}\;\;\lambda\in (-\frac{\pi}{2}, +\infty),
\end{aligned}
\end{cases}
\end{equation}
Thus, by denoting the  stationary profiles in \eqref{antikink} as a function of $\lambda$, $\Pi_{a_1(\lambda)}=(\phi_{1, a_1(\lambda)}, \phi_{2, a_2(\lambda)}, \phi_{2, a_2(\lambda)})$  (with $a_2(\lambda)=-a_1(\lambda)$) also represents a real-analytic mapping. Moreover,  $\Pi_{a_1(\lambda)}\to \Pi_{0}$ as $\lambda\to 0$ with $\Pi_{0}=(\phi_{1, a_1(0)}, \phi_{2, a_2(0)}, \phi_{2, a_2(0)})$ in the sense that $\|\phi_{j, a_2(\lambda)}-\phi_{j, a_2(0}\|_{H^1(0, +\infty)} \to 0$ as $\lambda\to 0$ for $j=2,3$, and $
\lim_{\lambda\to 0}\|\phi_{1, a_1(\lambda)}-\phi_{1, a_1(0)}\|_{H^1(-\infty, 0)}=0$.
Thus, we obtain that  $\mathcal L_{\lambda}$ converges to $\mathcal L_{0}$ as $\lambda\to 0$ in the generalized sense. Indeed, denoting $
W_\lambda=\Big( \cos (\phi_{j, a_1(\lambda)})\delta_{j,k}\Big)$
we obtain 
\begin{equation*}
 \begin{split}
 \widehat{\delta}(\mathcal L_\lambda, \mathcal L_{0} )&=\widehat{\delta}(\mathcal L_0 + (W_\lambda-W_{0}), \mathcal L_{0})\leqq \|W_\lambda-W_{0}\|_{L^2 (\mathcal Y)}\to 0,\qquad\text{as}\;\;\lambda\to 0,
\end{split}
\end{equation*}
where $\widehat{\delta}$ is  the gap metric  (see \cite[Chapter IV]{kato}). 

Now, we denote by $N=n(\mathcal L_{0})$ the Morse-index for $\mathcal L_{0}$. Thus, from Proposition \ref{antimain} we can separate the spectrum $\sigma(\mathcal L_{0})$ of $\mathcal L_{0}$   into two parts: $\sigma_0=\{\gamma: \gamma<0\}\cap \sigma (\mathcal L_{0})$ and $\sigma_1$ by a closed curve  $\Gamma$ belongs to the resolvent set of $\mathcal L_{0}$ with $0\in \Gamma$ and   such that $\sigma_0$ belongs to the inner domain of $\Gamma$ and $\sigma_1$ to the outer domain of $\Gamma$. Moreover, $\sigma_1\subset [\theta_{0}, +\infty)$ with $\theta_{0}=\inf\{\theta: \theta\in \sigma(\mathcal L_0),\; \theta>0\}>0$ (we recall that $\sigma_{\mathrm{ess}}=[1,+\infty)$). Then, by standard perturbation theory (see Kato \cite[Theorem 3.16, Chapter IV]{kato}), we have that $\Gamma\subset \rho(\mathcal L_{\lambda})$ for $\lambda\in [-\delta_1, \delta_1]$ and $\delta_1>0$ small enough. Moreover, $\sigma(\mathcal L_{\lambda})$ is likewise separated by $\Gamma$ into two parts so that the part of $\sigma (\mathcal L_{\lambda})$ inside $\Gamma$ consists of negative eigenvalues with exactly total (algebraic) multiplicity equal to $N$. Therefore, by  Proposition \ref{anti5main}, $n(\mathcal L_{\lambda})=N=1$ for $\lambda\in [-\delta_1, \delta_1]$.

Next, we use a classical continuation argument based on the Riesz-projection for showing that $n( \mathcal{L}_\lambda )=1$ for all $\lambda\in  (0, +\infty)$. Indeed, define $\omega$ by
$$
\omega= \sup \left\{\eta: \eta>0 \; \text{such that } \; n(\mathcal{L}_\lambda)=1\;\;\text{for all}\;\;\lambda\in (0, \eta)\right\}.
$$
The analysis above implies that $\omega$ is well defined, and $\omega>0$. We claim that $\omega=+\infty$.  Suppose that $\omega<+\infty$.  Let $M=n(\mathcal{L}_\omega)$, and $\Gamma$ be  a closed curve such that $0\in  \Gamma\subset \rho(\mathcal L_\omega)$, and all the negative eigenvalues  of $\mathcal L_\omega$ belong to the inner domain of $\Gamma$. Next,  using that as a function of $\lambda$, $(\mathcal{L}_\lambda)$ is a real-analytic family of self-adjoint operators of type $(B)$ in the sense of Kato  (see \cite{kato})  we deduce that there is  $\epsilon>0$ such that for $\lambda\in [\omega-\epsilon, \omega+\epsilon]$ we have $\Gamma\subset \rho(\mathcal L_{\lambda})$,  and  the mapping $\lambda\to (\mathcal L_{\lambda}-\xi)^{-1}$ is analytic for  $\xi\in \Gamma$. Therefore, the existence of an analytic family of Riesz-projections $\lambda\to P(\lambda)$ given by
$$
P(\lambda)=-\frac{1}{2\pi i}\ointctrclockwise_{\Gamma} (\mathcal L_\lambda-\xi)^{-1}d\xi 
$$
implies that $\dim(\mathrm{range} \, P(\lambda))=\dim(\mathrm{range} \, P(\omega))=M$ for all $\lambda\in [\omega-\epsilon, \omega+\epsilon]$. Further, by definition of $\omega$, there is $\eta_0\in (\omega-\epsilon,\omega)$ and $\mathcal{L}_\lambda$ has $n(\mathcal{L}_\lambda)=1$ for all $\lambda\in (0,\eta_0)$. Therefore, $M=1$ and $\mathcal L_{\omega+\epsilon}$ has exactly one negative eigenvalue, hence $\mathcal L_{\lambda}$ has $n(\mathcal{L}_\lambda)=1$  for $\lambda\in (0, \omega+\epsilon)$, which contradicts with the definition of $\omega$. Thus, $\omega= +\infty$. This finishes the proof.
\end{proof}

\begin{proof}[Proof of Theorem \ref{0antimain}]
From Propositions \ref{antimain}, \ref{anti5main} and \ref{antimain4} we have $\ker( \mathcal{L}_\lambda)=\{\bold{0}\}$ and $n( \mathcal{L}_\lambda)=1$. Moreover, Theorem \ref{cauchy2} verifies Assumption $(S_1)$ in the linear instability criterion in subsection 3.1. Thus, from Theorem \ref{crit} follows the linear instability property of the stationary kink/anti-kink soliton profile $\Pi_{\lambda, \delta'}$. Lastly, following the same strategy for showing Theorem \ref{well2} we obtain the local well-posedness theory for \eqref{akmodel} in $H^1(\mathcal Y)\times L^2(\mathcal Y)$ and so the analysis of Henry {\it et al.} in \cite{HPW82} (see also Angulo {\it{et al.}} \cite{ALN08,AngNat16}) we obtain the nonlinear instability property of $\Pi_{\lambda, \delta'}$. This finishes the proof.
\end{proof}


\section*{Acknowledgements}
J. Angulo  was supported in part by CNPq/Brazil Grant  and FAPERJ/Brazil program PRONEX-E - 26/010.001258/2016.  The work of R. G. Plaza was partially supported by DGAPA-UNAM, program PAPIIT, grant IN-100318.

\appendix
\section{Extension Theory}
\label{secAppET}


For the sake of completeness, in this appendix we develop the extension theory of symmetric operators suitable for our needs. For further information on the subject the reader is referred to the monographs by Naimark \cite{Nai67,Nai68}. The following classical result, known as  \emph{the von-Neumann decomposition theorem}, can be found in \cite{RS2}, p. 138.
 
\begin{theorem}\label{d5} 
Let $A$ be a closed, symmetric operator, then
\begin{equation}\label{d6}
D(A^*)=D(A)\oplus\mathcal N_{-i} \oplus\mathcal N_{+i}.
\end{equation}
with $\mathcal N_{\pm i}= \ker (A^*\mp iI)$. Therefore, for $u\in D(A^*)$ and $u=x+y+z\in D(A)\oplus\mathcal N_{-i} \oplus\mathcal N_{+i}$,
\begin{equation}\label{d6a}
A^*u=Ax+(-i)y+iz.
\end{equation}
\end{theorem}

\begin{remark} The direct sum in (\ref{d6}) is not necessarily orthogonal.
\end{remark}

The following propositions provide a strategy for estimating the Morse-index of the self-adjoint extensions and can be found in Naimark \cite{Nai68} (Theorem 16, p. 44) and Reed and Simon, vol. 2, \cite{RS2} (see Theorem X.2, p. 140).

\begin{proposition}\label{semibounded}
Let $A$  be a densely defined lower semi-bounded symmetric operator (that is, $A\geq mI$)  with finite deficiency indices, $n_{\pm}(A)=k<\infty$,  in the Hilbert space $\mathcal{H}$, and let $\widehat{A}$ be a self-adjoint extension of $A$.  Then the spectrum of $\widehat{A}$  in $(-\infty, m)$ is discrete and  consists of, at most, $k$  eigenvalues counting multiplicities.
\end{proposition}

\begin{proposition}\label{11}
	Let $A$ be a densely defined, closed, symmetric operator in some Hilbert space $H$ with deficiency indices equal  $n_{\pm}(A)=1$. All self-adjoint extensions $A_\theta$ of $A$ may be parametrized by a real parameter $\theta\in [0,2\pi)$ where
	\begin{equation*}
	\begin{split}
	D(A_\theta)&=\{x+c\phi_+ + \zeta e^{i\theta}\phi_{-}: x\in D(A), \zeta \in \mathbb C\},\\
	A_\theta (x + \zeta \phi_+ + \zeta e^{i\theta}\phi_{-})&= Ax+i \zeta \phi_+ - i \zeta e^{i\theta}\phi_{-},
	\end{split}
	\end{equation*}
	with $A^*\phi_{\pm}=\pm i \phi_{\pm}$, and $\|\phi_+\|=\|\phi_-\|$.
\end{proposition}

Next Proposition can be found in Naimark \cite{Nai68} (see Theorem 9, p. 38).

\begin{proposition}
\label{esse}
All self-adjoint extensions of a closed, symmetric operator
which has equal and finite deficiency indices have one and the
same continuous spectrum.
\end{proposition}

The following result was used in the proof of Proposition  \ref{5main}.

\begin{proposition}\label{2M}
Let $\mathcal{Y}$ be a $\mathcal{Y}$-junction of type I. Consider the following closed symmetric operator, $(\mathcal{H}, D(\mathcal{H}))$, densely defined on $L^2(\mathcal Y)$ by
\begin{equation}\label{3M}
\begin{split}
&\mathcal{H}=\Big (\Big(-c_j^2\frac{d^2}{dx^2}\Big)\delta_{j,k} \Big ),\;\;1\leqq j, k\leqq 3,\\
&D(\mathcal{H})= \Big \{(v_j)_{j=1}^3\in H^2(\mathcal{Y}):
c_1v_1'(0-)=c_2v'_2(0+)=c_3v_3'(0+)=0,\;\; \sum_{j=2}^3 c_jv_j(0+)-c_1v_1(0-)=0 \Big \},
\end{split}
\end{equation}
with $c_j>0$, $1\leqq j\leqq 3$, and $\delta_{j,k}$ being the Kronecker symbol. Then, the deficiency indices are $n_{\pm}( \mathcal H)=1$. Therefore, we have that all the self-adjoint extensions of $( \mathcal H, D( \mathcal H))$ are given by the one-parameter family $(\mathcal H_\lambda, D(\mathcal H_\lambda))$,  $\lambda\in \mathbb R$, with $\mathcal H_\lambda\equiv \mathcal H$ and $D(\mathcal H_\lambda)$  defined by
 \begin{equation}\label{deridelta}
D(\mathcal H_\lambda)= \Big\{(u_j)_{j=1}^3\in H^2(\mathcal{Y}):
c_1u'_1(0-)=c_2u'_2(0+)=c_3u'_3(0+),\;\; \sum_{j=2}^3 c_ju_j(0+)-c_1u_1(0-)=\lambda c_1u'_1(0-) \Big\}.
\end{equation}
\end{proposition}

\begin{proof}
We show initially  that the adjoint operator $( \mathcal H^*, D( \mathcal H^*))$ of $( \mathcal H, D( \mathcal H))$   is given by 
  \begin{equation}\label{2F*3}
 \mathcal H^*=\mathcal H, \quad D( \mathcal H^*)=\{(u_j)_{j=1}^ 3\in H^2(\mathcal Y) : c_1u'_1(0-)=c_2u'_2(0+)=c_3u'_3(0+)\}.
 \end{equation}
Indeed, formally for $\bold u=(u_1, u_2, u_3), \bold v=(v_1, v_2, v_3)\in H^2(\mathcal Y)$ we have 
\begin{equation}\label{2relaself}
\begin{split}
 \langle\mathcal H\bold u, \bold v\rangle
 &=c_1^2u_1(0-)v_1' (0-)- c_1^2u'_1(0-)v_1 (0-) + \sum_{j=2}^ 3c_j^2[u'_{j}(0+)v_j(0+) -u_{j}(0+)v'_j(0+)]+\langle \bold u,\mathcal H \bold v\rangle\\
 &\equiv R+ \langle \bold u,\mathcal H \bold v\rangle.
 \end{split}
 \end{equation}
Hence, for $\bold u, \bold v \in D(\mathcal H)$ we obtain $R=0$ in \eqref{2relaself} and so the symmetric property of $\mathcal H$. Next, let us denote $D_0^*:=\{(u_j)_{j=1}^ 3\in H^2(\mathcal G) : c_1u'_1(0-)=c_2u'_2(0+)=c_3u'_3(0+)\}$. We shall show that $D_0^*=D( \mathcal H^*)$. Indeed, from \eqref{2relaself}  we obtain for
 $\bold v\in D_0^*$ and $\bold u\in D( \mathcal H)$ that  $R=0$  and so $\bold v\in D (\mathcal H^*)$ with $\mathcal H^*\bold v=  \mathcal H\bold v$. Let us show the inclusion $D_0^*\supseteq
D(\mathcal H^*)$. Take $\bold u=(u_1, u_2, u_3)\in D(\mathcal H)$, then
for any $\bold v=(v_1, v_2, v_3)\in D(\mathcal H^*)$ we have from \eqref{2relaself}
\begin{align}\label{relaself3}
 \langle\mathcal H \bold u, \bold v\rangle = 
R+\langle \bold u,\mathcal H^* \bold v\rangle=\langle \bold u,\mathcal H \bold v\rangle.
 \end{align}
Thus, we arrive for any $\bold u\in D(\mathcal H)$ at the equality 
\begin{equation}\label{2adjoint}
 \sum_{j=2}^ 3c_j^2v'_{j}(0+)u_j(0+) - c_1^2v_1' (0-)u_1(0-)=0.
 \end{equation} 
Next, let us consider $\bold u=(u_1, u_2, 0)\in D(\mathcal H)$. Then $c_1 u_1'(0-)=c_2 u_2'(0+)=0$ and from equations \eqref{3M} and \eqref{2adjoint}, we obtain $
[c_2v'_2(0+)-c_1v'_1(0-)]c_1 u_1(0-)=0$.
Thus, by choosing $u_1(0-)\neq 0$ follows $c_1v'_1(0-)=c_2v'_2(0+)$. Similarly, we have $c_1v'_1(0-)=c_3v'_3(0+)$. Therefore,  $\bold v\in D_0^*$. This shows the relations in \eqref{2F*3}.

Now, from \eqref{2F*3} we obtain that the deficiency indices  for $( \mathcal H, D( \mathcal H))$ are $n_{\pm}( \mathcal H)=1$. Indeed,  $\ker(\mathcal H^*\pm iI)=[\Phi_{\pm}]$ with $\Phi_{\pm}$ defined by
\begin{equation}
\Phi_{\pm}=
\left(\underset{x<0}{e^{\frac{ ik_{\mp}}{c_1}x}}, \underset{x>0}{-e^{\frac{- ik_{\mp}}{c_2}x}}, \underset{x>0}{-e^{\frac{ -ik_{\mp}}{c_3}x}}\right),
\end{equation}
$k^2_{\mp}=\mp i$, $\IM(k_{-})<0$ and $\IM(k_{+})<0$. Moreover, $\|\Psi_{-}\|_{L^ 2(\mathcal Y)}=\|\Psi_{+}\|_{L^ 2(\mathcal Y)}$.

  Next, from extension theory for symmetric operators we have that every self-adjoint extension $(\widehat{\mathcal H}, D(\widehat{\mathcal H}))$ of $(\mathcal H, D(\mathcal H))$ is characterized by $D(\widehat{\mathcal H})\subset 
D(\mathcal H^*)$, and   $\bold u=(u_1,u_2,u_3)\in D(\widehat{\mathcal H})$ if and only if 
$
\sum_{j=2}^3 c_ju_j(0+)-c_1u_1(0-)=\lambda c_1u'_{1}(0-),
$
where, for $\theta\in [0, 2\pi)-\{\pi/2\}$,
$$
\lambda= -\Big(\sum_{j=1}^3 c_j \Big)\frac{1+e^{i\theta}}{e^{-i\frac{\pi}{4}}+e^{i(\theta+\frac{\pi}{4})}} \in \mathbb R.
$$
Thus, the set of self-adjoint extensions
$(\widehat{\mathcal H}, D(\widehat{\mathcal H}))$ of the symmetric operator $(\mathcal H, D(\mathcal H))$ can be seen as one-parametrized  family $(\mathcal H_\lambda, D(\mathcal H_\lambda))$ defined by $\mathcal H_\lambda=\mathcal H$ and $D(\mathcal H_\lambda)$ by \eqref{deridelta}. This finishes the proof.
 \end{proof}

\section{Morse index calculations}
\label{secMorse}
The following result gives a precise value for the Morse-index of the operator $ \mathcal{L}_{\lambda} $ in \eqref{deriva1} associated to the kink/anti-kink profile $(\phi_j)$ obtained in subsection \S \ref{secAK-K}, when we consider the  domain $\mathcal C_1\cap D (\mathcal{L}_{\lambda})$, $\mathcal C_1=\{(u_j)_{j=1}^3\in L^2(\mathcal Y): u_2=u_3\}$. We establish this result in order to possible future study related to the stability properties of kink/anti-kink profiles. We conjecture that these profiles are in fact unstable (see Remark \ref{extension2} and \cite{AngCav}).

\begin{proposition}\label{Appanti}
Let $\lambda\in (-\infty, -\frac{\pi}{2})$ and consider  $\mathcal B_2=\mathcal C_1 \cap D (\mathcal{L}_{\lambda})$.  Then $\mathcal{L}_{\lambda}  : \mathcal B_2\to \mathcal C_1$ is well defined and $n( \mathcal{L}_{\lambda} |_{\mathcal B_2} )=2$.
\end{proposition}

\begin{proof} The proof is based on analytic perturbation theory.  By Proposition \ref{antimain}  we  have  for $\Phi'_{\lambda_0}=(2\phi'_1, \phi'_2,  \phi'_2)\in \mathcal C_1$, $\phi_i\equiv  \phi_{i, \lambda_0}$, that $\Span\{\Phi'_{\lambda_0}\}=\ker(\mathcal{L}_{\lambda_0} |_{\mathcal B_2} )$ ($\lambda_0\equiv -\frac{\pi}{2}$). Moreover, by the proof of Proposition \ref{anti5main} we have for $\Lambda_1=(\phi'_1, \phi'_2,  \phi'_2)\in \mathcal C_1\cap H^1(\mathcal Y)$ that the quadratic form $\mathcal Q$    in \eqref{antiqua} satisfies $\mathcal Q(\Lambda_1)<0$ thus $n(\mathcal{L}_{\lambda_0} |_{\mathcal B_2} )=1$.  Let us denote by $\chi_{_{\lambda_0}}$ the unique negative eigenvalue for $\mathcal{L}_{\lambda_0}|_{\mathcal B_2}$. We note  in this point  of the analysis that by using the classical perturbation theory and the convergence  $\mathcal{L}_{\lambda} \to \mathcal{L}_{\lambda_0}$ as $\lambda \to \lambda_0$  in the generalized sense  (Kato \cite{kato}) we obtain that given a closed curve $\Gamma$ such that $\sigma_0=\{\chi_{_{\lambda_0}}, 0\}$ belongs to the inner domain of $\Gamma$, then we can only conclude that $ n(\mathcal{L}_{\lambda} |_{\mathcal B_2} )\geqq 1$ (by Proposition \ref{antimain}) for $\lambda\approx \lambda_0$. So   we will need to determine exactly how the eigenvalue zero will move, either to right or to left . In this form, we divide our analysis into several steps:

\begin{enumerate}
 \item[i)]  The mapping $\lambda\in (-\infty, \infty)\to \Psi_\lambda= (\phi_{i, a_i(\lambda)})$ is real-analytic and we have the convergence $ \Psi_\lambda- \Psi_{\lambda_0}\to 0$, as $\lambda\to \lambda_0$, in $H^1(\mathcal Y)$. Then, it follows that $\mathcal{L}_{\lambda}$ converges to $\mathcal{L}_{\lambda_0}$ as $\lambda \to \lambda_0$  in the generalized sense (see proof of Proposition \ref{antimain4}). Moreover, the   family  $\{\mathcal{L}_\lambda\}_{\lambda\in I}$ represents a  real-analytic family of self-adjoint operators of type (B) in the sense of Kato (see \cite{kato}) on $\mathcal C_1$. Therefore, from  Theorem IV-3.16 from Kato \cite{kato}  we obtain   $\Gamma\subset \rho(\mathcal{L}_{\lambda})$ for $|\lambda-\lambda_0|$ sufficiently small, and $\sigma (\mathcal{L}_{\lambda})$ is likewise separated by $\Gamma$ into two parts, such   that the part of $\sigma (\mathcal{L}_{\lambda})$ inside $\Gamma$ consists of a finite number of eigenvalues with total multiplicity (algebraic) equal to two. Then, it follows from  Kato-Rellich's theorem (see \cite{RS1}) the existence of two analytic functions, $\Omega$ and $\Pi$, defined in a neighborhood of $\lambda_0$ with $\Omega: (\lambda_0-\delta, \lambda_0+\delta)\to \mathbb R$ and $\Pi: (\lambda_0-\delta, \lambda_0+\delta)\to \mathcal C_1$ such that
$\Omega(\lambda_0)=0$ and $\Pi(\lambda_0)=\Phi'_{\lambda_0}$. For all $\lambda\in (\lambda_0-\delta, \lambda_0+\delta)$, $\Omega(\lambda)$ is the  simple isolated second eigenvalue of $\mathcal{L}_{\lambda}$, and $\Pi(\lambda)$ is the associated eigenvector for $\Omega(\lambda)$. Moreover, $\delta$ can be chosen small enough to ensure that, for  $\lambda\in (\lambda_0-\delta, \lambda_0+\delta)$,  the spectrum of $\mathcal{L}_{\lambda}$ in $\mathcal C_1$ is positive, except for, at most, the first two eigenvalues.

\item[ii)] Using Taylor's theorem, we obtain for sufficiently   small $\delta$  the following expansions:  \begin{equation}\label{antitaylor1}
\Omega(\lambda)=\beta (\lambda-\lambda_0)+ O(|\lambda-\lambda_0|^2), \;\; \Pi(\lambda)=\Phi'_{\lambda_0} + \Pi'(\lambda_0) (\lambda-\lambda_0) + \bold{O}(|\lambda-\lambda_0|^2),
\end{equation}
where $\beta= \Omega'(\lambda_0)$. Moreover, by analyticity we also obtain the expansion
 \begin{equation}\label{antitaylor2}
\Psi_\lambda-\Psi_{\lambda_0}=(\lambda-\lambda_0) \partial_\lambda \Psi_\lambda |_{\lambda=\lambda_0} + \bold{O}(|\lambda-\lambda_0|^2).
\end{equation}

Next, we determine the sign of $\beta$.  Thus, we compute $\langle \mathcal{L}_{\lambda} \Pi(\lambda), \Phi'_{\lambda_0} \rangle$ in two different ways. On one hand, using $\mathcal{L}_{\lambda} \Pi(\lambda)= \Omega(\lambda)\Pi(\lambda)$, \eqref{antitaylor1} leads to
 \begin{equation}\label{antitaylor3}
\langle \mathcal{L}_{\lambda} \Pi(\lambda), \Phi'_{\lambda_0}
 \rangle= \beta (\lambda-\lambda_0)\|\Phi'_{\lambda_0}\|^2 + O(|\lambda-\lambda_0|^2).
\end{equation}
On the other hand, since $\Phi'_{\lambda_0}\in D(\mathcal{L}_{\lambda} )\cap \mathcal C_1$ for all $\lambda\neq \lambda_0$ and $\mathcal{L}_{\lambda} $ is self-adjoint, we get $\langle \mathcal{L}_{\lambda} \Pi(\lambda), \Phi'_{\lambda_0}  \rangle=\langle \Pi(\lambda), \mathcal{L}_{\lambda} \Phi'_{\lambda_0} \rangle$. Now, using the notation $\phi_{i,\lambda}=\phi_{i,a_i(\lambda)}$, we have from \eqref{deriva1} the relation $
\mathcal{L}_{\lambda} \Phi'_{\lambda_0}  = \mathcal{L}_{\lambda_0} \Phi'_{\lambda_0}  +\mathcal A \Phi'_{\lambda_0} =\mathcal A \Phi'_{\lambda_0}$
where $\mathcal A$ is the diagonal-matrix $
\mathcal A=\Big((\cos(\phi_{i,\lambda})-\cos(\phi_{i,\lambda_0}))\delta_{i,k}\Big)$, $ 1\leqq i,k\leqq 3$.
Next, for $G_i(\lambda)=\cos(\phi_{i,\lambda})$ we have the expansion
$$
G_i(\lambda)-G_i(\lambda_0)=-\sin(\phi_{i,\lambda_0})\Big( \partial_\lambda \phi_{i,\lambda}|_{_{\lambda=\lambda_0}}\Big)(\lambda-\lambda_0)+ O(|\lambda-\lambda_0|^2).
$$
Then $
\mathcal A=\Big((-\sin(\phi_{i,\lambda_0})\partial_\lambda \phi_{i,\lambda}|_{_{\lambda=\lambda_0}}(\lambda-\lambda_0))\delta_{i,k}\Big) + \bold{O}(|\lambda-\lambda_0|^2)$, $1\leqq i,k\leqq 3$.
Thus,
\begin{equation}\label{antitaylor5}
\begin{split}
\langle \mathcal{L}_{\lambda} \Pi(\lambda), \Phi'_{\lambda_0}  \rangle&=\langle \Pi(\lambda), \mathcal{L}_{\lambda} \Phi'_{\lambda_0}  \rangle=\langle \Phi'_{\lambda_0}  + \Pi'(\lambda_0) (\lambda-\lambda_0) + \bold{O}(|\lambda-\lambda_0|^2), \mathcal A \Phi'_{\lambda_0} \rangle
\\
&=\Big\langle \Phi'_{\lambda_0}, \Big((-\sin(\phi_{i,\lambda_0})\partial_\lambda \phi_{i,\lambda}|_{_{\lambda=\lambda_0}}(\lambda-\lambda_0))\delta_{i,k}\Big)\Phi'_{\lambda_0}  \Big \rangle+ O(|\lambda-\lambda_0|^2)\\
&\equiv (\lambda-\lambda_0) \eta + O(|\lambda-\lambda_0|^2),
\end{split}
\end{equation}
with $
\eta\equiv 4a_1'(\lambda_0) \int_{-\infty}^0 (\phi'_{1,\lambda_0})^3\sin(\phi_{1,\lambda_0})\, dx+2a_1'(\lambda_0) \int_0^{\infty} (\phi'_{2,\lambda_0})^3\sin(-\phi_{2,\lambda_0}) \, dx$.
Next, we obtain the sign of $\eta$. Indeed, from the qualitative properties of the anti-kink profile $\phi_{i,\lambda_0}$ we get immediately that the two integrals above are positive. Next, from the relation in \eqref{anticondi}  we obtain that $a_1'(\lambda_0)=\frac13$ (indeed we always have $a_1'(\lambda)>0$). Then $\eta>0$. Next, from \eqref{antitaylor3} and \eqref{antitaylor5} there follows $
\beta=\frac{\eta}{\|\Phi'_{\lambda_0}\|^2} + O(|\lambda-\lambda_0|)$.
Therefore, from \eqref{antitaylor1} there exists $\delta_0>0$ sufficiently small such that $\Omega(\lambda)>0$ for any $\lambda\in (\lambda_0, \lambda_0+\delta_0)$, and $\Omega(\lambda)<0$ for  any $\lambda\in (\lambda_0-\delta_0, \lambda_0)$.  Thus, in the space  $ \mathcal C_1$, the Morse index $n(\mathcal{L}_\lambda |_{\mathcal B_2})=1$ for $\lambda>\lambda_0$ and $\lambda\approx \lambda_0$ (equality that is used in Proposition \ref{anti5main}), and $n(\mathcal{L}_\lambda |_{\mathcal B_2})=2$ for $\lambda<\lambda_0$ and $\lambda\approx \lambda_0$.

\item[iii)] Next, since   $\ker(\mathcal{L}_\lambda)=\{\bold 0\}$ for $\lambda\neq -\frac{\pi}{2}$, we obtain via a continuation argument based on the Riesz-projection operator (see proof of Proposition \ref{antimain4}) that for any $\lambda\in (-\infty, -\frac{\pi}{2})$ follows $n(\mathcal{L}_\lambda |_{\mathcal B_2})=2$ . This finishes the proof.
\end{enumerate}
\end{proof}

\def\cprime{$'\!\!$} \def\cprimel{$'\!$}

%

\end{document}